\documentclass{amsart}

\usepackage{graphicx}
\usepackage{enumerate}
\usepackage{amsmath,amssymb}
\usepackage{mathrsfs}

\usepackage{color}

\newtheorem{theorem}{Theorem}[section]
\newtheorem{lemma}[theorem]{Lemma}
\newtheorem{proposition}[theorem]{Proposition}
\newtheorem{claim}[theorem]{Claim}
\newtheorem{corollary}[theorem]{Corollary}
\newtheorem{remark}[theorem]{Remark}
\newtheorem{definition}[theorem]{Definition}

\newtheorem*{caveat*}{Caveat}

\numberwithin{equation}{section}
\numberwithin{figure}{section}

\def\ind{\mathrm{index}}
\def\loc{\mathrm{loc}}
\def\Diff{\mathrm{Diff}}
\def\Cl{\mathrm{Cl}}

\begin{document}

\title[
Blenders,  cu-H\'enon-like families and  heterodimensional  bifurcations
]
{
Blenders in center unstable H\'enon-like families:  with an application to 
 heterodimensional bifurcations
}

\author{
Lorenzo J. D\'iaz
}
\address{
Depto.\ Matem\'atica, PUC-Rio, Marqu\^es de S.\ Vicente 225 22453-900 Rio de Janeiro RJ  Brazil
}
\email{
lodiaz@mat.puc-rio.br
}

\author{
Shin Kiriki
}
\address{
Dept. of Mathematics, Tokai University,
4-1-1 Kitakaname, Hiratuka Kanagawa, 259-1292, Japan
}
\email{
kiriki@tokai-u.jp
}

\author{
Katsutoshi Shinohara
}
\address{
FIRST, Aihara Innovative Mathematical Modelling Project, 
   Institute of Industrial Science,
   The University of Tokyo, 
   4-6-1 Komaba, Meguro-ku, Tokyo 153-8505, Japan}
\email{
herrsinon@07.alumni.u-tokyo.ac.jp
}


\date{\today}


\begin{abstract}
We give an explicit family of polynomial maps called center unstable H\'enon-like maps and prove that they exhibits blenders
for some parametervalues. 
Using this family, we also prove the occurrence of blenders near certain non-transverse heterodimensional cycles under high regularity assumptions. 
The proof involves a renormalization
scheme along heteroclinic orbits. We also investigate the connection 
between the blender and the original heterodimensional cycle.
\end{abstract}
\maketitle

\section{Introduction}\label{s.Introduction}
This paper has two main goals. The first one is to exhibit an explicit family of quadratic polynomial maps
in dimension three (center unstable H\'enon-like families) with blenders.
The second one is to prove the occurrence of blenders near certain non-transverse heterodimensional cycles and the 
connections between them under high regularity assumptions. 
The two previous results are related as follows: associated to the non-transverse heterodimensional cycles there are renormalization schemes converging to center unstable H\'enon-like families. 

We now briefly discuss the three main topics of this paper: 
blenders, non-transverse heterodimensional cycles, and renormalization. In what follows
we  assume that the dimension of the ambient space is three.

\subsection*{Blenders and H\'enon-like families}
{\emph{Blenders}} (see Definition~\ref{d.blender})
only appear in dimension greater than or equal to $3$
and are just a special type of hyperbolic sets $\Lambda$  of diffeomorphisms
$f$ which are maximal
invariant in a neighborhood $\Delta$, that is, 
$\Lambda=\cap_{n\in \mathbb{Z}} f^n(\Delta)$. 
We consider the case where the stable direction is one-dimensional.
Then the blender $\Lambda$ has a dominated
splitting with three hyperbolic directions (the stable, the center unstable, and the
strong unstable directions).
A key property of a blender is its internal dynamical
configuration that implies that every curve which crosses 
a distinctive open region of $\Delta$ and is almost tangent to the 
one-dimensional strong unstable direction 
intersects the local stable manifold of $\Lambda$.
This roughly means that  $\Lambda$ topologically behaves 
as a hyperbolic set with stable manifold 
of dimension two. 
Another relevant
property of a blender is its robustness: for diffeomorphisms $g$ near $f$
the continuation of the hyperbolic set $\Lambda$ for $g$ is also a blender.

A blender is an important ingredient
for obtaining robust non-hyperbolic dynamics: blenders 
play a similar role as the thick horseshoes
introduced by Newhouse \cite{Nthick,N79}.
They are used  to turn heterodimensional cycles and
homoclinic tangencies $C^1$-robust, see \cite{BD08,BD12}.
We remark that
recently some authors have introduced blenders 
whose  center unstable direction is indecomposable and has dimension $\geq 2$,
 see \cite{Enrique,Yuri}.

We study how  blenders occur. 
As far as we know, their construction involves series of 
perturbations which are genuinely $C^{1}$, see for instance  \cite{BD08,BD12}.
We present an explicit family of quadratic maps 
(which we call  a \emph{center unstable H\'{e}non-like family}) with blenders.
A novelty here is that the blenders are obtained without perturbations
and their occurrence only involves an appropriate selection of
parameters of the family.

\begin{theorem}\label{thm.main2}
Consider the center unstable H\'enon-like family
\begin{equation}
\label{e.henonfamily}
G_{\xi, \mu,\kappa, \eta}
({x}, {y}, {z})=
(\xi {x}+  {y},\ 
{\mu}+{y}^{2} + \kappa  {x}^{2}+ \eta {x}{y},\ 
{y}), \quad \xi>1.
\end{equation}
There is an open set $\mathcal{B}$
of parameters $(\xi,\mu,\kappa,\eta)$
  such that
any diffeomorphism $F$ $C^{r}$-close to
$G_{\xi, \mu,\kappa, \eta}$ with
$(\xi, {\mu},\kappa, \eta)\in \mathcal{B}$
has a blender.
\end{theorem}

One important feature of the family (\ref{e.henonfamily}) is the existence 
of the term $y$ in the $x$-component. 
This term gives us a superposition property (see Definition~\ref{d.blender}) of the hyperbolic set
which enables us to obtain the blender, 
compare the family treated in \cite{GL10}.
\smallskip

%
%
%
%
%

\subsection*{Coexistence of critical and noncritical dynamics}
Homoclinic tangencies and H\'enon-like dynamics are in the core
of the so-called \emph{critical dynamics}, while heterodimensional cycles 
are genuine bifurcations of  \emph{non-critical dynamics}, for more details see
Preface of \cite{BDV}. 
There are many cases
where the effects of the critical and the non-critical dynamics overlap: the system 
has a critical region and a non-critical one and there are transitions
between these two regions.
We study
such kind of configuration in dimension three: diffeomorphisms having
two saddles of 
different {\emph{indices}} (dimension of
the unstable direction)
whose invariant manifolds are cyclically related
by a
heterodimensional cycle with non-transverse heteroclinic intersections.
This configuration (depicted in Figure~\ref{fig1})
is called a
 {\emph{non-transverse  heterodimensional cycle}}: 
the two invariant manifolds of dimension one meet quasi-transversally and 
the two-dimensional ones have a tangential intersection.
The dynamics close to the saddles
and the quasi-transverse heteroclinic intersection provide the non-critical part
of the dynamics, while the critical one is given by the heteroclinic tangency. 

\begin{figure}[hbt]
\centering
\scalebox{0.5}{\includegraphics[clip]{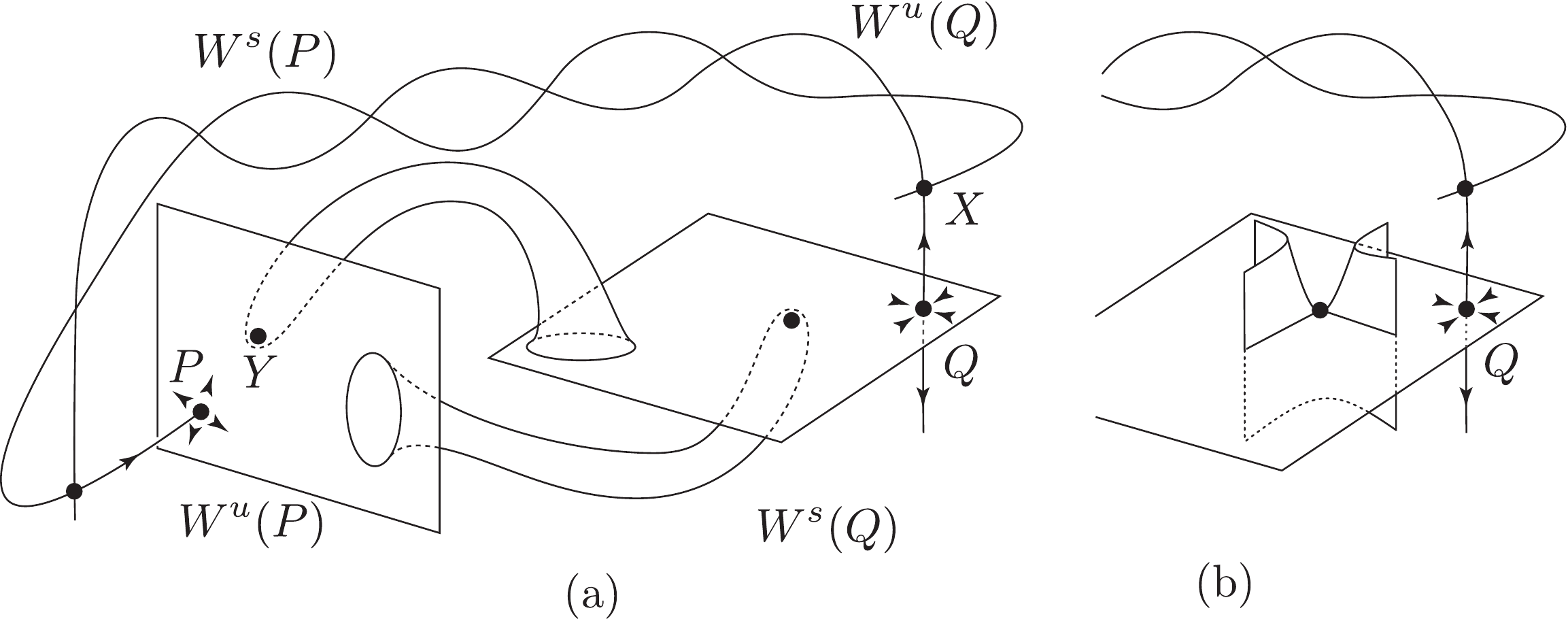}}
\caption{
Non-transverse heterodimensional cycles
}
\label{fig1}
\end{figure}

The non-transverse cycle that we consider contains a heterodimensional tangency.
In the $C^1$ case, it is known that 
the unfolding of these tangencies leads to robustly non-dominated dynamics and in some
cases to very intermingled dynamics 
related to universal dynamics introduced in \cite{BD03},
see \cite{DNP06} for detail. 
In the $C^2$ case, the non-transverse cycles 
are \emph{not} special in the following sense:
there exists $C^{2}$ open setw of $n$-dimensional ($n>2$) diffeomorphisms 
having heterodimensional tangencies in a coindex-$(n-2)$ heterodimensional cycle, 
see \cite{KS12}. Note that the occurence of $C^{2}$ Newhouse phenomenon  
and strange attractors were observed in \cite{KNS10}.

Using Theorem~\ref{thm.main2} and the renormalization scheme
that we will present below, we prove
the existence of blenders
near some of these non-transverse heterodimensional cycles
under high regularity assumptions, see
Corollary~\ref{cor.kiriki}. Let us observe that by adjusting the parameters,
we can obtain convergence of the renormalization to different types of H\'enon-like
maps, see the comment after Theorem~\ref{thm.main1}. 
This illustrates the richness of the dynamics near these cycles.

\subsection*{Renormalization}
Renormalization means providing
a sequence of local coordinate changes near a 
tangency point and
reparametrizations, which gives a sequence of
return maps along heteroclinic orbits converging to a limit map with interesting 
dynamics. Using a renormalization scheme one can  translate relevant properties of the limit maps
to some diffeomorphisms close to the one with the tangency.

Renormalization methods play an important role in the study of
homoclinic bifurcations (dynamics at homoclinic tangencies). 
This method leads to the approximation of 
the dynamics by quadratic families and allows to translate properties of such families
(such as existence of strange attractors and sinks, or thick hyperbolic sets) to properties of the diffeomorphisms,
see for instance \cite[Chapter 3]{PT93}. 

So far, renormalizations have not 
been sufficiently exploited in the context of heterodimensional bifurcations.
We consider renormalization schemes
for non-transverse heterodimensional cycles. 
Depending on the conditions satisfied by the diffeomorphism, this renormalization 
may converge to 
different types of dynamics. 
We wonder if this limit map may exhibit blenders.
In general, it is not always the case.
We prove that, under some (degenerate) conditions, 
the renormalization converges to the center unstable H\'{e}non-like families 
in Theorem~\ref{thm.main2}, see Theorem~\ref{thm.main1}.

We now state precise definitions and results.

\subsection*{Definitions and statements of the results}
Consider diffeomorphisms $f$ defined on a closed  
three-dimensional manifold $M$  having hyperbolic
periodic points $P$ and $Q$ of saddle type with different 
{\emph{indices}}
(dimension of the unstable bundle and denoted by $\ind(\cdot)$).
We assume that $f$ has a {\emph{non-transverse heterodimensional cycle}} 
associated to the saddles  $P$ and $Q$ (assume that $\ind(P)=2 >\ind (Q)=1$). 
This means that 
$$
 W^{s}(P)\cap W^{u}(Q)\neq \emptyset \quad \mbox{and} \quad
 W^{s}(Q)\cap W^{u}(P)\neq \emptyset,
 $$
 where
 the first intersection between one-dimensional manifolds is \emph{quasi-transverse},  that is, there is 
 $X\in W^s(P) \cap W^u(Q)$ such that
\begin{equation}\label{e.quaset}
\dim (T_{X}W^{s}(P)+T_{X}W^{u}(Q))=
\dim (E^s(P)) + \dim (E^u(Q)), 
\end{equation}
and the two-dimensional manifolds 
$W^{u}(P)$ and  $W^{s}(Q)$ have a non-transverse intersection
(a tangency)
along the orbit of a point $Y$.
This geometrical configuration  is depicted in
 Figure \ref{fig1} (the tangency in (a) is 
\emph{elliptic} and the one in (b) is  \emph{hyperbolic}, see \cite{KNS10}).

Associated to the heteroclinic point $X$,
there is a transition
from a neighborhood $U_Q$  of $Q$ to a neighborhood
$U_P$ of $P$ following the orbit of $X$. 
Similarly, associated to $Y$ 
there is a transition 
from $U_P$ to $U_Q$, see also Figure~\ref{fig4}. 
We impose some conditions to these transitions 
in equation \eqref{tech-assumption1a}.

Let
$\mathrm{per}(P)$ and $\mathrm{per}(Q)$
be the periods of $P$ and $Q$, respectively.
Denote the eigenvalues of 
$Df^{\mathrm{per}(P)}(P)$ and $Df^{\mathrm{per}(Q)}(Q)$
by $\tilde{\lambda}$, $\tilde{\zeta}$, $\tilde{\sigma}$ and
$\lambda$, $\zeta$, $\sigma$, respectively.
We assume that
\begin{equation} \label{eigenvalue condition 1}
 |\tilde{\lambda}|< 1<|\tilde{\sigma}|< |\tilde{\zeta}|,
\qquad
 |\lambda|< |\zeta|<1<|\sigma|
\end{equation} 
and 
\begin{equation}\label{eigenvalue condition 2}
\left| |\tilde{\sigma}\tilde{\zeta}|^{k} \sigma\zeta^{2} \right|,\ 
\left| |\tilde{\sigma}^{-3}\tilde{\zeta}|^{k} \sigma^{-1} \right|,\ 
\left| |\tilde{\lambda}\tilde{\sigma}|^{k} \sigma \right|
<1, \quad \mbox{where 
$k=\dfrac{\log\vert {\lambda} \vert^{-1}}{\log\vert \tilde\zeta \vert}$.} 
\end{equation}

As we mentioned, bifurcations of non-transverse heterodimensional
cycles exhibit a rich variety of dynamics.
Depending on the way of unfolding, one may reach different 
types of dynamics. 
In the following theorem,  starting from a six-parameter family of diffeomorphisms 
(where the parameters describe the position of the continuations of 
the heteroclinic points above),
we select a two-parameter sub-family  
converging to a H\'{e}non-like family:
there is a renormalization scheme near the tangency providing
a sequence of maps $\{ F_{\mu_k (\bar \mu),\nu_k} \}$, 
repa\-ra\-me\-tri\-za\-tions $\mu_{k}$, and parameters $\nu_{k}$ such that 
the family converges to a H\'{e}non-like family.

\begin{theorem}\label{thm.main1}
Consider  $f\in \Diff^{r}(M)$, $\dim(M)=3$ and
$r\geq 2$,
with a non-transverse heterodimensional cycle
associated with saddles
$P$ and $Q$ with heteroclinic orbits 
$X$ (quasi-transverse point) and $Y$ (tangency point). 

Assume that 

\begin{itemize}
\item The local dynamics around $P$ and $Q$ are linearized as  in  (\ref{eq.linearization}), 
\item the eigenvalues of $Df^{\mathrm{per}(P)}(P)$ and $Df^{\mathrm{per}(Q)}(Q)$
satisfy  (\ref{eigenvalue condition 1}) and (\ref{eigenvalue condition 2}),
\item the transitions between $P$ and $Q$ satisfy
(\ref{tech-assumption1a}),(\ref{tech-assumption2}) and (\ref{non-zero}).
\end{itemize}
Then there is a six-parameter family $\{f_{\mu,\nu} \}_{\mu, \nu\in [-\epsilon,\epsilon]^3} \subset \Diff^{r}(M)$
 with $f_{\mathbf{0},\mathbf{0}}=f$,
satisfying the following:
for any real number $\xi > 0$,
there are
\begin{itemize}
\item
 a sequence of coordinate changes
  $\Psi_{k}: K \to M$ 
 defined
  near the tangency
point $Y$,
where $K$ is a compact neighborhood of the origin of  $\mathbb{R}^3$
that can be taken arbitrarily large,
\item
reparametrizations 
$
\mu_k \colon I \to \mathbb{R}^3,
$
where $I$ is a compact neighborhood of $0$ in $\mathbb{R}$ that can be taken arbitrarily large,
\item parameter values
$\nu_{k}\in \mathbb{R}^{3}$, 
\item
a sequence of pairs of 
natural numbers $(m_k, n_k)_k$, $m_k, n_k\to \infty$, and
natural numbers
$N_{1}$ and $N_{2}$ independent of $k$ and $\xi$,
\end{itemize}
such that the map 
\begin{equation}\label{eq.renormalization}
F_{\mu_{k}(\bar{\mu}),\nu_{k}}(\bar {\mathbf{x}})=
(\Psi_{k}^{-1}|_{\Psi_{k}(K)} )
\circ f^{N_{2}+m_k+N_{1}+n_k}_{\mu_{k}(\bar{\mu}),\nu_{k}}
\circ \Psi_{k}(\bar{\mathbf{x}})
\end{equation}
satisfies the following properties:  
\begin{enumerate}[{(}1{)}]
\item \label{thm.main1(1)}
Suppose $(\bar{\mu}, \bar{\mathbf{x}})$ is contained in 
a (fixed) compact set of $\mathbb{R}\times \mathbb{R}^3$,
then  the sequence
$\{(\mu_{k}(\bar{\mu}), \nu_{k})\}$ converges
to  $(\mathbf{0},\mathbf{0})\in \mathbb{R}^6$ 
and the sequence $\{\Psi_{k}(\bar{\mathbf{x}})\}$ converges to 
the tangency point $Y$ as $k \to +\infty$.
\item \label{thm.main1(2)}
The sequence of maps $\{ F_{{\mu_{k}(\bar{\mu}),\nu_{k}}} \}$
converges in the $C^{r}$ topology  to
a one-parameter family conjugate to 
\begin{equation}\label{limit return map}
G_{ \mu}({x}, {y}, {z})=
G_{\xi, \mu,\kappa_{1}, \kappa_{2}}({x}, {y}, {z})=
(\xi {x}+  {y},\ 
{\mu}+{y}^{2} + \kappa_{1} {x}^{2}+ \kappa_{2} {x} {y},\
y),
\end{equation}
where $\kappa_{1}, \kappa_{2}$ are constants depending only on $f$.
\end{enumerate}
\end{theorem}

An immediate consequence of this theorem and Theorem~\ref{thm.main2} 
is the following:

\begin{corollary}\label{cor.kiriki}
Let $f$ be a $C^{r}$ diffeomorphisms ($r \geq 2$) satisfying the hypotheses of 
Theorem~\ref{thm.main1}. Then every $C^{r}$-neighborhood of $f$ 
contains an open set of diffeomorphisms having blenders.
\end{corollary}

Let us give a comment to Theorem~\ref{thm.main1}.  
The numbers $(m_k, n_{k})$ correspond to 
the consecutive times during which              
the points stay close to the saddles $P$ and $Q$, respectively.
The selection of these numbers determines convergence of the renormalization
and the number $\xi$. This choice may lead to $\xi>1$ 
(center unstable H\'enon-like maps)
or $\xi<1$ (center stable H\'enon-like maps). This means that arbitrarily
close to the original system there are both types of dynamics.

\vskip .2cm

\vskip .2cm

\subsection*{Connecting one-dimensional invariant manifolds}
We now show an application of the methods 
above. 
To exploit completely the consequence of our techniques is 
beyond the goal of this paper (it is a part of on-going research project).

Theorem \ref{thm.main1} asserts 
that arbitrarily $C^r$-close to certain types of diffeomorphisms $f$ 
having a non-transverse heterodimensional cycle,
there exist diffeomorphisms $g$ 
having center unstable blenders $\Lambda_g$
near the point of tangency.
The renormalization gives us some local information, 
but does not provide information about the 
connections  between $\Lambda_g$ and
the continuations $P_g$, $Q_g$ of  $P$, $Q$ (semi-global information). 
Since the dynamics is not dominated close to the tangency point, 
it is not easy to describe the relative positions of the invariant manifolds 
of the these sets. 

Under some additional hypotheses on the contracting multipliers of $Q$, 
the next theorem gives the creation of 
 robust intersections between the ``one-dimensional'' 
invariant manifolds $W^{s}(\Lambda_{g}, g)$  and $W^u(Q_g,g)$.
This intersection is in principle harder to obtain 
than the one between the ``two-dimensional'' invariant manifolds.

\begin{theorem}\label{thm.main3}
Let $f$ be a diffeomorphism 
having a non-transverse heterodimensional 
cycle
associated to saddles $P$ and $Q$ 
satisfying the hypothesis of
Theorem \ref{thm.main1}. 
Consider a real number $\alpha >0 $ satisfying 
$\alpha < \frac{\log|\lambda|}{\log|\zeta|} -1$, 
where $|\lambda| <  |\zeta|$ are the moduli of contracting eigenvalues of 
$Df^{\mathrm{per}(Q)}(Q)$ in (\ref{eigenvalue condition 1}). 
Then, for every $r \geq 2$,
arbitrarily $C^{1+\alpha}$ close to $f$,
there exists a $C^{r}$ diffeomorphism $g$ 
having a center-unstable blender $\Lambda_{g}$ such that
$W_{\loc}^s(\Lambda_{g}, g) \cap W^u(Q_{g}, g) \neq \emptyset$ holds $C^{1}$-robustly. 
\end{theorem}

Note that the $C^{1}$-robustness in the conclusion of 
Theorem~\ref{thm.main3} implies the $C^{r}$-robustness
of the connection for all $r > 1$.
The number $\alpha$ in the theorem is in $(0, 1/2)$, 
see Lemma \ref{lem.Katsutoshi} in Section \ref{ss.eieva}.

In the $C^{1}$-topology, the connecting lemma of Hayashi \cite{H97}
and the constructions in \cite{BD08} 
tell us that there are diffeomorphisms arbitrarily $C^{1}$-close to $f$ with robust cycles.
We wonder if this is possible to obtain such connections for systems $C^{r}$-close to $f$ 
for $r > 1$.
We use a detailed estimate on the hyperbolic behavior of the limit map obtained in the renormalization process 
to get robust intersections for $C^{1+\alpha}$-approximations.
As $\alpha$ is in $(0, 1/2)$, we do not obtain, for example,  a $C^{2}$-result.

The above result gives a connection from $Q_{g}$ to $\Lambda_{g}$.
However, in general it is also difficult to get
non-empty intersections between
$W^s(Q_{g}, g)$ and $W^u(\Lambda_{g}, g)$. This difficulty is due to the 
fact that the transition map from $P$ to $Q$
does not preserves the ``central direction''.
Thus, the question still remains whether or not a $C^r$-robust heterodimensional cycle
can be created from a bifurcation of the heterodimensional tangency.

To get such intersections one can consider two possible directions. 
The first one is to get a renormalization scheme providing 
further geometrical information of the cycles.
This strategy is well exploited in the renormalizations in \cite{PV94, R94}, 
where the dynamics is normally hyperbolic (which is not our case).
Another direction is to consider heterodimensional tangencies
as in \cite{DNP06} associated to robustly non-hyperbolic transitive sets. 
These sets have rich structures that may help to find such intersections.

\medskip

\medskip

\noindent{{\bf{Organization of the paper.}}}
In Section~\ref{s.Blender-Henon families}
we introduce the definition of a blender, state a result guaranteeing its existence,
and prove Theorem~\ref{thm.main2} about the occurrence of blenders for some
H\'enon-like families.
In Section \ref{s.Transitions} we describe the class of non-transverse heterodimensional cycles we consider. This description involves properties of the transitions between the saddles in the cycle following the non-transverse heteroclinic orbits and the local
dynamics near the saddles.
The six-parameter family of diffeomorphisms in  Theorem~\ref{thm.main1} is
presented in Section~\ref{ss.Local perturbations}.
In Section \ref{s.Proof of Theorem1}  
we introduce the renormalization scheme and prove
Theorem \ref{thm.main1}.
Finally, 
in Section \ref{s.proof-of-last-thm}, 
Theorem \ref{thm.main3}  is shown by previous theorems
and a certain $C^{1+\alpha}$ 
perturbation which is performed locally near the local unstable manifold of the continuation of $Q$.

\section{Blenders  for center unstable H\'enon-like maps}\label{s.Blender-Henon families}
In this section, we consider diffeomorphisms which are $C^{r}$-near
the center unstable H\'enon-like  endomorphism 
 \begin{equation}\label{e.new}
G_{\mu,\kappa,\xi}(x, y, z)= (y,  \mu+y^{2}+\kappa z^{2},    \xi z+y).
\end{equation}
This map is conjugate to the  H\'enon family in (\ref{e.henonfamily})
when $\eta=0$
by the coordinate change
\begin{equation}\label{coordinate change2}
\tilde{\Theta}: (x,y,z)\longmapsto(z,y,x).
\end{equation}
We will prove the existence of blenders for diffeomorphisms close to
$G_{\mu, \kappa, \xi}$ where the parameters are in some specific ranges
(see Theorem~\ref{thm.main2}).

\subsection{Conditions for the existence of blenders}\label{ss.Definitions and distinctive property}
Before going to the proof of the theorem,  
let us recall the definition of a blender and 
 sufficient conditions for the existence of blenders
in \cite[\S 1]{BD96}. 
First, we give an axiomatic definition of a blender. 

\begin{definition}[Blender, Definition 3.1 in \cite{BD12}]
\label{d.blender} {\em{Let $f\colon M\to M$ be a diffeomorphism. A
transitive  hyperbolic compact set $\Lambda$ of $f$ with 
index $k$, $k\ge 2$, is a {\emph{$cu$-blender}} if there are a
$C^1$-neighborhood $\mathcal{U}$ of $f$ and a $C^1$-open set $\mathcal{D}$ 
of embeddings of $(k-1)$-dimensional disks $D$ into $M$ such that
 for every
$g \in \mathcal{U}$ and
every disk $D\in \mathcal{D}$ the local stable manifold
$W^s_{\loc}(\Lambda_g)$ of $\Lambda_g$ (the continuation 
of $\Lambda$ for $g$) 
intersects $D$.
The set $\mathcal{D}$ is called the {\emph{superposition}}
set of the blender.}}
\end{definition}

We now give some preliminary definitions that we borrow from \cite{BD96}.
Let   $\Delta=I_{x}\times I_{y}\times I_{z}$ be a cube in $\mathbb{R}^{3}$ where 
$I_{x}=[x^{-}, x^{+}]$, $I_{y}=[y^{-}, y^{+}]$ and $I_{z}=[z^{-}, z^{+}]$ are intervals.
Divide the boundary of $\Delta$ into three parts as follows:
$$
\partial^{ss}\!\Delta:=  \partial I_{x}\times  I_{y}\times I_{z},\
\partial^{uu}\!\Delta:=I_{x}\times \partial I_{y}\times  I_{z},\
\partial^{u}\!\Delta:=I_{x}\times \partial(I_{y}\times I_{z}).
$$
Note that $\partial^{uu}\!\Delta\subset\partial^{u}\!\Delta$. For $\theta>1$ let
 $\mathcal{C}_{\theta}^{u}$, $\mathcal{C}_{\theta}^{uu}$, and $\mathcal{C}_{\theta}^{s}$ be
 cone fields defined as follows: for $p\in \Delta$, put 
 \begin{equation}\label{eqn.cones}
\begin{split}
&\mathcal{C}_{\theta}^{u}(p)=\left\{(u,v,w)\in T_{p} \Delta\ \bigr\vert\    \theta |u|\leqslant \sqrt{v^{2}+w^{2}}  \right\},\\[0.15cm]
&\mathcal{C}_{\theta}^{uu}(p)=\left\{(u,v,w)\in T_{p} \Delta\ \bigr\vert\    \theta\sqrt{u^{2}+w^{2}} \leqslant  |v| \right \},\\[0.15cm]
&\mathcal{C}_{\theta}^{s}(p)=\left\{(u,v,w)\in T_{p} \Delta\ \bigr\vert\    \theta\sqrt{v^{2}+w^{2}} \leqslant  |u| \right \}.
\end{split}
 \end{equation}
Note that $\mathcal{C}_{\theta}^{uu}(p)\subset \mathcal{C}_{\theta}^{u}(p)$.

Then we define as follows:
\begin{itemize}
\item
A regular curve $L\subset \Delta$ is  \emph{vertical}
if $T_{p} L\subset \mathcal{C}_{\theta}^{uu}(p)$ for every point $p$ in $L$, 
and the end-points of $L$ are contained in different connected components of $\partial^{uu}\!\Delta$.
\item
A surface $S\subset \Delta$ is called a   {\emph{vertical strip in $\Delta$}}
if
$T_{p} S\subset \mathcal{C}_{\theta}^{u}(p)$  for every $p$ in $S$
and there exists
a $C^{1}$ embedding $E:I_{y}\times J \rightarrow \Delta$ 
(where $J$ is a subinterval of $I_{z}$) such that
$L(z):=E(I_{y}\times\{z\} )$ is a vertical  curve  for every  $z\in  J$.
The width of $S$, denoted by $w(S)$, is the infimum of the length of the  
curves in $S$ which are
transverse to $\mathcal{C}_{\theta}^{uu}$ joining the two boundary components of  $L(\partial J)$.
\end{itemize}

Let $W$ be a curve in $\Delta$ tangent to the cone field $\mathcal{C}^{s}$
whose  endpoints  are contained in different connected components of  $\partial^{ss}\!\Delta$. Note that there are two different homotopy classes of vertical segments
through $\Delta$ disjoint from $W$. 

\begin{itemize}
\item
A vertical curve  $L$ in  $\Delta$ with $L\cap W=\emptyset$ 
is \emph{to the right of} $W$ if it is in the homotopy class of
$\{x^{-}\}\times I_{y}\times \{z^{+}\}$ for some $x_{0}\in I_{x}$.  
Similarly,
a vertical strip $S$ through $\Delta$ is \emph{to the right of} $W$
if any vertical curve in $S$ is to the right of $W$.
\end{itemize}

For  a three dimensional diffeomorphism $F$,
the next geometric conditions (H1)--(H5)
guarantee the  existence of a blender in $\Delta$, see \cite[\S 1]{BD96}:
\begin{enumerate}[{(H}1{)}]
\item There is a connected component $A$ of $\Delta\cap F(\Delta)$ 
disjoint from $\partial^{ss}\!\Delta\cup F(\partial^{u}\!\Delta)$.
\item There is a connected
component $B$ of $F(\Delta)\cap \Delta$ such that $B$ 
is disjoint
from $I_{x}\times I_{y}\times \{z^{+}\}$, from
 $\partial^{ss}\!\Delta$ and from $F(\partial^{uu}\!\Delta)$.
\item 
There are  $\theta>1$ and $\ell\in \mathbb{N}$ 
such that
the cone fields $\mathcal{C}_{\theta}^{u}$, $\mathcal{C}_{\theta}^{uu}$, 
and $\mathcal{C}_{\theta}^{s}$  
satisfy the following conditions: There is $c>1$ such that
\begin{enumerate}[{(}i{)}]
\item
For every $p\in F^{-1}(A \cup B)$ such that $F^i(p) \in \Delta$ for every
$i=0,\dots,\ell-1$ 
and every  $\mathbf{v}\in \mathcal{C}_{\theta}^{u}(p) \setminus \{\bf{0} \}$,
$(DF^i)_{p} \mathbf{v}$ belongs to the interior of
$\mathcal{C}_{\theta}^{u}(F^i(p))$
 and
$|(DF^\ell)_{p} \mathbf{v}|\geqslant c |\mathbf{v}|$.
\item
For every $p\in F^{-1}(A\cup B)$ and
every $\mathbf{v}\in \mathcal{C}_{\theta}^{uu}(p) \setminus \{\bf{0} \}$,
 $\mathbf{w}:=(DF)_{p} \mathbf{v}$ belongs to the interior of
$\mathcal{C}_{\theta}^{uu}(F(p)) $.

\item For every $p\in A\cup B$ and
every $\mathbf{v}\in \mathcal{C}_{\theta}^{s}(p) \setminus \{\bf{0} \}$,
 $\mathbf{w}:=(DF^{-1})_{p} \mathbf{v}$
 belongs to the interior of
$\mathcal{C}_{\theta}^{s}(F^{-1}(p))$  and $|\mathbf{w}|\geqslant c |\mathbf{v}|$.
\end{enumerate}
\end{enumerate}

Note that conditions (H1) and (H3) imply 
that $F$ has a (unique) hyperbolic fixed point $P_\ast$ 
in $A$ with index one. 
Let 
$W^{s}_{0}$ be the connected component
of $W^{s}( P_\ast)\cap \Delta$ containing $P_\ast$.
Observe that
the curve $W^{s}_{0}$ is tangent to the cone field $\mathcal{C}^{s}_\theta$.
So we can speak of vertical curves and strips 
being to the left or to the right of $W^{s}_{0}$.

\begin{itemize}
\item[(H4)]
There is a neighborhood $U^{-}$ of the left side $\{z=z^{-}\}$ of $\Delta$
so that every vertical strip $S$ through $\Delta$ to the right of
$W^{s}_{0}$ does not intersect $U^{-}$.

\item[(H5)]
There exist neighborhoods $U$  of  $W^{s}_{0}$
and  $U^{+}$ of the right side $\{z=z^{+}\}$ of $\Delta$  such that
for  every vertical strip $S$ through $\Delta$ to the right of $W^{s}_{0}$ 
one of the two following possibilities holds:
\begin{enumerate}[(i)]
\item The intersection $F(S)\cap A$ contains a vertical strip
$\Sigma$ through $\Delta$ to the right of $W^{s}_{0}$
and disjoint from $U^{+}$;
\item $F(S)\cap B$ contains a vertical strip $\Sigma$ through
$\Delta$ to the right of $W^{s}_{0}$ and disjoint from $U$.
\end{enumerate}
\end{itemize}

Note that the presentation of (H3) is slightly different from the one in \cite{BD96}.  
In our (H5), 
each non-zero vector of 
$\mathcal{C}^{uu}_\theta$ is expanded only after $\ell$ iterations. 
However, by some standard argument, one can check that the conditions
(H1)--(H5) above are also sufficient to guarantee the occurrence of the blender. 
Indeed, these conditions imply that the width of vertical strips 
in $\Delta$ 
grows  exponentially after iterations
by $F^\ell$. This implies that
 the stable manifold $W^{s}(P_\ast)$ of $P_\ast$ intersects transversally
 every vertical strip $S$ through $\Delta$ to the right of $W^{s}_{0}$, see 
 \cite[Lemma 1.8]{BD96}.
 In particular, for 
$\Lambda=\bigcap_{i \in \mathbb{Z}} F^{i}(\Delta)$
we have $W^s_{\mathrm{loc}} (\Lambda) \cap S\ne \emptyset$. Thus
$\Lambda$
is a blender whose superposition set is formed by 
the vertical segments through $\Delta$  to the right of $W^{s}_{0}$.
  
\subsection{Proof of Theorem \ref{thm.main2}} \label{ss.Proof of Theorem 1.2}
We consider the
open set of parameters $\mathcal{O}$
and the cube $\Delta$ in $\mathbb{R}^{3}$ defined as follows:
\begin{equation}\label{parameters/box for blenders}
\begin{split}
& \mathcal{O} =\left\{(\mu,\kappa,\xi)\ \vert\ 
-10<\mu<-9,\
0<\kappa<10^{-4},\
1.18<\xi<1.19 \right\},\\
& \Delta = \left\{ (x,y,z) \ \vert\ 
 |x|,|y|\leqslant 4,\ -40\leqslant z\leqslant 0\right\}.
\end{split}
\end{equation}
The boundary $\partial \Delta$ of $\Delta$ is divided into two parts
$\partial^{u} \Delta :=[-4, 4]\times \partial([-4,4]\times [-40,0])$ and 
$\partial^{ss} \Delta  :=\Delta \cap \{x=\pm 4\}$. Moreover, 
consider the subset 
$\partial^{uu} \Delta :=\Delta \cap \{y=\pm 4\}$ 
of $\partial^{u} \Delta$. 
The goal of this section is to prove the following proposition. 

\begin{proposition}\label{prop.blender}
There is $\varepsilon>0$ such that every diffeomorphism
 $F$  
sufficiently close to $G_{\mu,\kappa,\xi,\eta}$
(see \eqref{e.henonfamily} for the definition of $G_{\mu,\kappa,\xi,\eta}$) where
 $(\mu,\kappa,\xi,\eta)\in \mathcal{O}\times (-\varepsilon, \varepsilon)$ 
 has a blender.
\end{proposition}
This proposition immediately implies Theorem \ref{thm.main2}. 
It follows from the  next three lemmas and the sufficient conditions for blenders
in Section~\ref{ss.Definitions and distinctive property}.
So, let us start the proof of the proposition above.

\begin{lemma}\label{l.H1H2}
If a three-dimensional diffeomorphism $F$  is
sufficiently $C^1$-close to $G_{\mu,\kappa,\xi}$ with $(\mu,\kappa,\xi)\in \mathcal{O}$,
then there are compact subsets $A$ and $B$ of $\Delta$ 
 satisfying (H1) and  (H2) for $F$.
\end{lemma}
\begin{proof}
First, we investigate these properties for
$G_{\mu, \kappa, \xi}$ with $(\mu, \kappa, \xi) \in \mathcal{O}$.

Let
$\Pi_{z}: \mathbb{R}^{3}\to \mathbb{R}^{2}$
and
$\Pi_{x}: \mathbb{R}^{3}\to \mathbb{R}^{2}$
be the projections 
$\Pi_{z}(x,y,z)=(x,y)$ and $\Pi_{x}(x,y,z)=(y,z)$.
We first observe the relation between $\Pi_{z}(\Delta)$ and $G_{\mu,\kappa,\xi}(\Delta)$.
For $z\in [-40, 0]$,
write 
$$D_{z}=[-4,4]^{2}\times \{z\}, \quad 
\partial^{uu}D_{z}=[-4,4]\times\{-4, 4\}\times \{z\}.
$$
Note that,
$G_{\mu,\kappa,\xi}$ is an endomorphism such that 
$DG_{\mu,\kappa,\xi}$ has a zero eigenvalue 
whose eigenspace is the $x$-axis.
Thus, for every $z\in [-40, 0]$, from \eqref{e.new} we have
$$
\Pi_{z}( G_{\mu,\kappa,\xi}(D_{z}))=
\left\{(x,y)\ \vert\ y=\mu+x^{2} +\kappa z^{2},\ |x|\leqslant 4 \right\}.
$$

Consider the sets
$$
D_{z}^{+}:=D_{z}\cap \{y>0\}, \quad
\mbox{and}
\quad
D_{z}^{-}:=D_{z}\cap \{y<0\}.
$$

Since $-10<\mu<-9$ and $0< \kappa z^{2}<0.16$,  
by direct  calculations, 
one can obtain  the following conditions, see Figure \ref{fig2}-(a):

\begin{itemize}
\item
$\Pi_{z} ( G_{\mu,\kappa,\xi}(D_{z}^{+}))\cap \Pi_{z} (\Delta)$
contains a segment $I^{+}_{z}$ transverse
in the $xy$-plane
 to $\Pi_{z}(\partial^{uu}\Delta)$
and
such that
 $I^{+}_{z}\subset  \{(x,y) \ \vert \ 2.4<x<3.8\}$.
\item
$\Pi_{z} ( G_{\mu,\kappa,\xi}(D_{z}^{-}))\cap \Pi_{z} (\Delta)$
contains a segment $I^{-}_{z}$ transverse
in the $xy$-plane
 to $\Pi_{z}(\partial^{uu}\Delta)$
and
such that
 $I^{-}_{z}\subset  \{(x,y) \ \vert \ -3.8<x<-2.4\}$.
\item
We have $\Pi_{z}(G_{\mu,\kappa,\xi}( \partial^{uu} D_{z})) =\{ (\pm 4,\mu+16+\kappa z^{2} )\}$. 
By the conditions on the constants,
its $y$-coordinate $\mu+16+\kappa z^{2}$ is greater than $4$.
That is, the projection
$\Pi_{z} ( G_{\mu,\kappa,\xi}( \partial^{uu} D_{z}))$
is outside $\Pi_z(\Delta)$.
\end{itemize}

\begin{figure}[hbt]
\scalebox{0.8}{\includegraphics[clip]{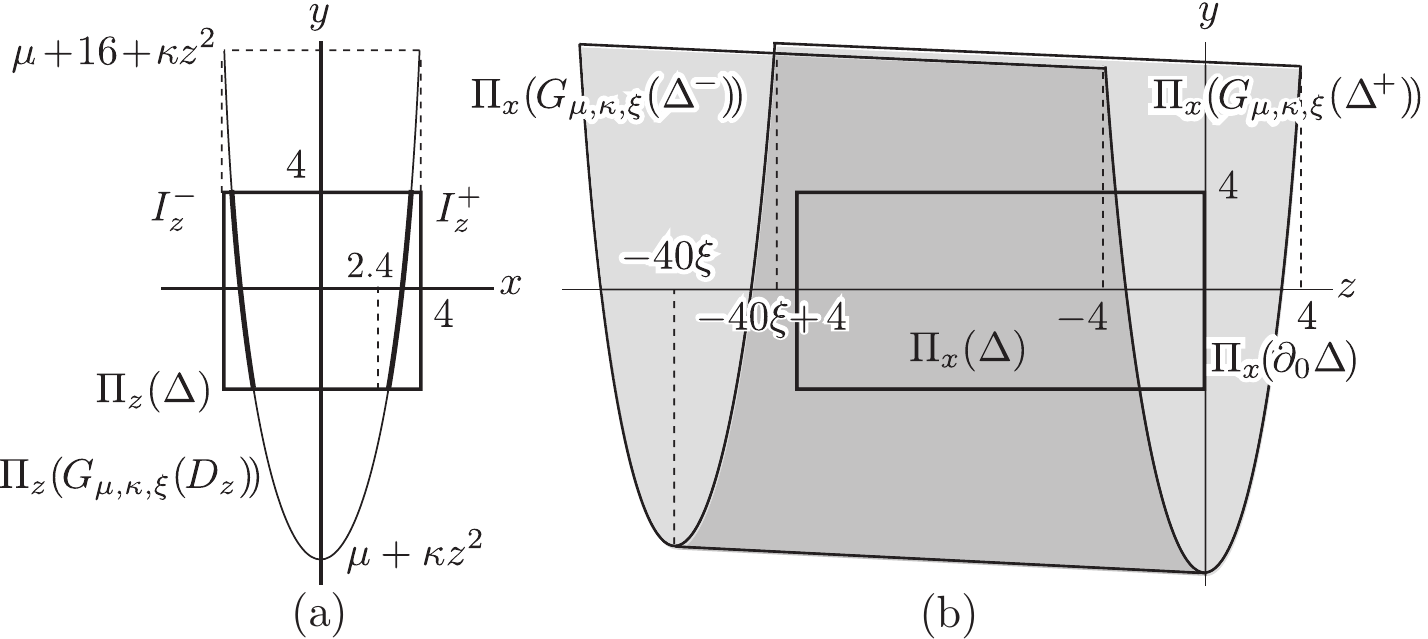}}
\caption{Projected images of 
 $\Delta$ and  $ G_{\mu,\kappa,\xi}(\Delta)$}
\label{fig2}
\end{figure}

Consider the sets
$$
\Delta^{+}:=\Delta\cap \{y>0\}, \quad
\Delta^{-}:=\Delta\cap \{y<0\},
\quad
\mbox{and}
\quad  
\partial_{0}\Delta:=\Delta\cap\{z=0\}.
$$
The following properties of
$\Pi_{x}(\Delta)$ and $\Pi_x(G_{\mu,\kappa,\xi}(\Delta))$
can be checked by direct calculations, see Figure \ref{fig2}-(b).

\begin{itemize}
\item $\Pi_{x}(G_{\mu,\kappa,\xi}(\Delta^{+}))$ contains
$\Pi_{x}(\Delta)=\left\{(z,y)\ \vert\  -40\leqslant z\leqslant 0,  |y|\leqslant 4\right \}$ 
in its interior.
\item
$\Pi_{x}(G_{\mu,\kappa,\xi}(\Delta^{-})) \cap \Pi_{x}(\Delta)$ contains
the set $\Pi_{x}(\Delta)\cap \{ -40< z< -4 \}$, and
\item
$\Pi_{x}(G_{\mu,\kappa,\xi}(\Delta^{\pm}))\cap
\Pi_{x}(\partial_{0}\Delta)=\emptyset$.
\end{itemize}
Note that 
$$
\Delta^{\pm}= \bigcup_{z\in [-40,0]}D_{z}^{\pm}.
$$ 
Consider now the sets  given by
$$
A= G_{\mu,\kappa,\xi}(\Delta^{+})\cap \Delta
\quad \mbox{and} \quad
B= G_{\mu,\kappa,\xi}(\Delta^{-})\cap \Delta.
$$
By the comments above,
$$
\Pi_{z}(A)=\bigcup_{z\in [-40, 0]}I^{+}_{z}
\quad
\mbox{and}
\quad
\Pi_{z}(B)= \bigcup_{z\in [-40, 0]}I^{-}_{z},
$$ 
and
$
\Pi_{x}(A) = \Pi_{x}(\Delta)
$
and
$
  \Pi_{x}(B) \cap \Pi_{x}(\partial_{0}\Delta)=\emptyset.
$
Moreover, from
the above observations, we have
\begin{itemize}
\item $A\cap  \partial^{ss}\Delta=\emptyset$ and  $A \cap  G_{\mu,\kappa,\xi}(\partial^{u}\Delta)=\emptyset$;
\item  $B\cap  \partial^{ss}\Delta=\emptyset$, $B\cap  \partial_{0}\Delta=\emptyset$,
and $B\cap  G_{\mu,\kappa,\xi}( \partial^{uu} D_{z})=\emptyset$.
\end{itemize}
This implies that $A$ and $B$
satisfy  (H1) and (H2).

Clearly,
any diffeomorphism $F$  
sufficiently $C^{1}$ close to $G_{\mu,\kappa,\xi}$ 
also satisfies these properties.
\end{proof}

\begin{lemma}\label{l.H3}
For any diffeomorphism $F$  sufficiently $C^{1}$-close to $G_{\mu,\kappa,\xi}$,
 $(\mu,\kappa,\xi)\in \mathcal{O}$,
 the cone fields
$\mathcal{C}_{2}^{u}$, $\mathcal{C}_{2}^{uu}$ and
$\mathcal{C}_{2}^{s}$
satisfy (H3) for $A, B$ as in Lemma~\ref{l.H1H2}.
\end{lemma}
\begin{proof}
First, note that for $p=(x,y,z)\in \Delta$ with $G_{\mu,\kappa,\xi}(p)\in A\cup B$
and $\mathbf{v}=(u,v,w) \in T_p\Delta$,
we have 
\begin{equation}\label{eq.derivative}
(u_{1},v_{1},w_{1}):=(DG_{\mu,\kappa,\xi})_{p} \mathbf{v}=(v, 2yv+2\kappa z w, v+\xi w).
\end{equation}
For the proof of (H3)-(i) and (ii), we just investigate 
the property of $G_{\mu, \kappa, \xi}$ with 
$(\mu, \kappa, \xi) \in \mathcal{O}$, which implies 
that the same conditions hold for $F$ near $G_{\mu, \kappa, \xi}$. 

\smallskip

\noindent{{\bf{Proof of (H3)-(i).}}}
Consider the cone
$$
\mathcal{C}^{u}(p):=\mathcal{C}_{2}^{u}(p)=\left\{
(u,v,w)\in T_{p}\mathbb{R}^{3}\
\vert\ 
 2 |u|\leqslant \sqrt{v^{2}+w^{2}}
\right\}.
$$
We consider the norm
$$
|(u,v,w)|_{*}:=
\max \left\{|u|, \sqrt{v^{2}+w^{2}}
\right\}.
$$
We will see that 
$|(DG_{\mu,\kappa,\xi})_{p}\mathbf{v}|_\ast> 
|\mathbf{v}|_\ast$ for every $\mathbf{v}\in \mathcal{C}^{uu}(p)\setminus \{\bf 0\}$. 
By compactness this implies that $|(DG_{\mu,\kappa,\xi})_{p}\mathbf{v}|_\ast> 
c_{0} |\mathbf{v}|_\ast$ for some uniform $c_{0}>1$. Since $|\cdot|_\ast$ is
equivalent to $|\cdot|$, this implies that there are $\ell$ and $c>0$ such that
$|(DG^\ell_{\mu,\kappa,\xi})_{p}\mathbf{v}|> 
c |\mathbf{v}|$. We now go to the details of the proof.

First, note that from the proof of Lemma \ref{l.H1H2},
the condition $G_{\mu,\kappa,\xi}(p)\in A\cup B$ implies $|y|>2.4$.
(remember that the $x$-coordinate of $G_{\mu,\kappa,\xi}(p)$ 
is equal to the $y$-coordinate of $p$).

We divide the proof into two cases:
$6.5|v|\geqslant |w|$ and $6.5|v|< |w|$.
In the case of $6.5|v|\geqslant |w|$, 
by this inequality and the choice of parameters in (\ref{parameters/box for blenders}),
we have 
$$
|v_{1}|=| 2yv+2\kappa z w |
\geqslant
2|y| |v|-2|\kappa| |z| |w|
>
4.8|v|-0.052|v|>4.7 |v|.
$$
Thus, 
\begin{equation}\label{e.newnew}
\sqrt{v_{1}^{2}+w_{1}^{2} }\geq |v_{1}|\geqslant 4.7|v|=4.7 |u_{1}|.
\end{equation}
Therefore, $(u_{1},v_{1},w_{1})\in \mathcal{C}^{u}(G_{\mu,\kappa,\xi}(p))$.

To get the uniform expansion of the vectors, note that
$$
w_{1}^{2}=(v+\xi w)^{2}\geqslant v^{2}-2\xi|v||w|+\xi^{2} w^{2}
\geqslant (1-13\xi) v^{2}+\xi^{2} w^{2},
$$ 
thus, together with \eqref{e.newnew}, one obtains 
$$
|(DG_{\mu,\kappa,\xi})_{p}\mathbf{v}|_{*}^{2}=
v_{1}^{2}+w_{1}^{2} \geqslant (22-13\xi) v^{2}+\xi^{2} w^{2}>4 v^{2}+1.18^{2} w^{2}
>|\mathbf{v}|_{*}^{2}.
$$

Next, we consider the case of $6.5|v|<|w|$.  Since
$$
|w_{1}|=| v+\xi w|\geqslant \xi |w|-|v|\geqslant 6.5\xi |v|-|v|>5|v|,
$$ 
one has
$$
\sqrt{v_{1}^{2}+w_{1}^{2} } \geqslant |w_{1}| >  5 |v|=5 |u_{1}|.
$$
This implies $(u_{1},v_{1},w_{1})\in \mathcal{C}^{u}(G_{\mu,\kappa,\xi}(p))$.
Moreover, by  (\ref{parameters/box for blenders}), (\ref{eq.derivative}) and  $2.4<| y|<3.8$,
\begin{eqnarray*}
&& v_{1}^{2}=( 2yv+2\kappa z w )^{2}\geqslant 4y^{2}v^{2}-8|yv\kappa z w|+4\kappa^{2} z^{2} w^{2}
>23  v^{2}-0.02 w^{2},\\
&& w_{1}^{2}=(v+\xi w)^{2}\geqslant \xi^{2} w^{2}-2\xi|v||w|+v^{2}
>1.026w^{2}+v^{2}.
\end{eqnarray*}
Therefore, 
$$|(DG_{\mu,\kappa,\xi})_{p}\mathbf{v}|_{*}^{2}=
v_{1}^{2}+w_{1}^{2}
> 24 v^{2}+1.006 w^{2}
> |\mathbf{v}|_{*}^{2}.$$
This completes the proof of (H3)-(i).
\smallskip

\noindent{{\bf{Proof of (H3)-(ii).}}}
For $p\in \Delta$ with $G_{\mu,\kappa,\xi}(p)\in A\cup B$,  take a cone
$$\mathcal{C}^{uu}(p):=\mathcal{C}_{2}^{uu}(p)=\left\{
(u,v,w)\in T_{p}\mathbb{R}^{3}\ \vert\ 2 \sqrt{u^{2}+w^{2}} \leqslant |v|
\right\}\subset  \mathcal{C}^{u}(p).
$$
This implies that 
$|w|\leq |v|/2$
for any $(u,v,w)\in \mathcal{C}^{uu}(p)$.
Recalling
(\ref{eq.derivative}) and  $2.4<| y|$, 
 we have
$$
|v_{1}|=|2yv+2\kappa zw|\geqslant 2|y||v|-2|\kappa||z||w|>4.7 |v|.
$$
On the other hand,
$$
u_{1}^{2}+w^{2}_{1}=v^{2}+(v+\xi w)^{2}\leqslant2v^{2}+2\xi|v||w|+\xi^{2}w^{2}\leqslant
2v^{2}+1.19 v^{2}+0.63 v^{2} < 4 v^{2}.
$$
Hence, 
$$
2 \sqrt{u_{1}^{2}+w^{2}_{1}}<4 |v|< |v_{1}|.
$$
Thus  $(DG_{\mu,\kappa,\xi})_{p} \mathbf{v}\in \mathcal{C}^{uu}(G_{\mu,\kappa,\xi}(p))$,
which completes the proof of (H3)-(ii).

\smallskip

\noindent{{\bf{Proof of (H3)-(iii).}}}
The existence of a contracting and invariant
strong stable cone field
follows from the fact that
$DG_{\mu,\kappa,\xi}$ is an endomorphism whose 
eigenspace associated the eigenvalue $0$ is spanned by $(1,0,0)$.
This implies that for every diffeomorphism $F$ sufficiently close to $G_{\mu,\kappa,\xi}$ with $(\mu,\kappa,\xi)\in \mathcal{O}$,
the cone field $\mathcal{C}_{2}^{s}(p)$ with $p\in  A\cup B$
satisfies (H3)-(iii).
\end{proof}

\begin{lemma} \label{l.H4H5}
Every diffeomorphism $F$ 
sufficiently $C^1$-close to $G_{\mu,\kappa,\xi}$, with $(\mu,\kappa,\xi)\in \mathcal{O}$,
satisfies conditions  (H4)  and (H5) for $A$, $B$ as in Lemma~\ref{l.H1H2}.
\end{lemma}
\begin{proof}
Again, we mainly consider 
the properties for $G_{\mu, \kappa, \xi}$ with 
$(\mu, \kappa, \xi) \in \mathcal{O}$. The conclusion for $F$ near 
$G_{\mu, \kappa, \xi}$ follows almost immediately. 

\noindent{\bf{Proof of (H4).}}
By construction, $A$ contains
unique saddle fixed point $P_{\ast}=(x_{\ast}, y_{\ast}, z_{\ast})$ 
which satisfies
$$
x_{\ast}=y_{\ast}=\mu+{y_{\ast}}^{2}+\kappa {z_{\ast}}^{2}=(1-\xi)z_{\ast}.
$$
Note that 
$W_{0}^{s}=\{(x, y_{\ast},z_{\ast})\ \vert\  |x|\leq 4\}$ 
and that every point in $W_0^s$ 
is  mapped to $P_{\ast}$ by  $G_{\mu,\kappa,\xi}$ by $G_{\mu, \kappa, \xi}$.
Therefore $W_{0}^{s}$ is a local stable manifold of $P_{\ast}$.
Since
$2.4<y_{\ast}<3.8$
and $1.18<\xi<1.19$,
one has the following estimation:
\begin{equation} \label{z-coordinate of P_{ast}}
-21.2<z_{\ast}=-\frac{y_{\ast}}{\xi-1}<-12.6.
\end{equation}
Note that the cones
$\mathcal{C}^{uu}(p)$ are defined
around the $y$-axis with slope $1/2$.
Thus a simple calculation gives that  any  vertical curve $L$, (i.e., 
a curve with $T_{p}L\subset \mathcal{C}^{uu}(p)$), through $\Delta$ to the right  $W^{s}_{0}$
does not intersect  a small neighborhood $U^{-}$ of  $\{z=-40\}\cap \Delta$.
This implies that (H4) holds for $G_{\mu,\kappa,\xi}$ and thus
for every diffeomorphism $F$ sufficiently $C^1$-close to $G_{\mu,\kappa,\xi}$.

\smallskip

\noindent{\bf{Proof of (H5).}}
Consider the subsets of $\Delta$ defined by 
\[
\begin{split}
A^{\prime}&=\left\{ (x,y,z)\ \vert\  |x|\leqslant 4,\  
2.4\leqslant y \leqslant 3.8,\    
-22 \leqslant z\leqslant  -3.3 \right\},\\
B^{\prime}&=\left\{ (x,y,z)\ \vert\  |x|\leqslant 4,\  
-3.8\leqslant y \leqslant  -2.4,\  
-7.3\leqslant z\leqslant 0   \right\}.
\end{split}
\]

\begin{figure}[hbt]
\centering
\scalebox{0.8}{\includegraphics[clip]{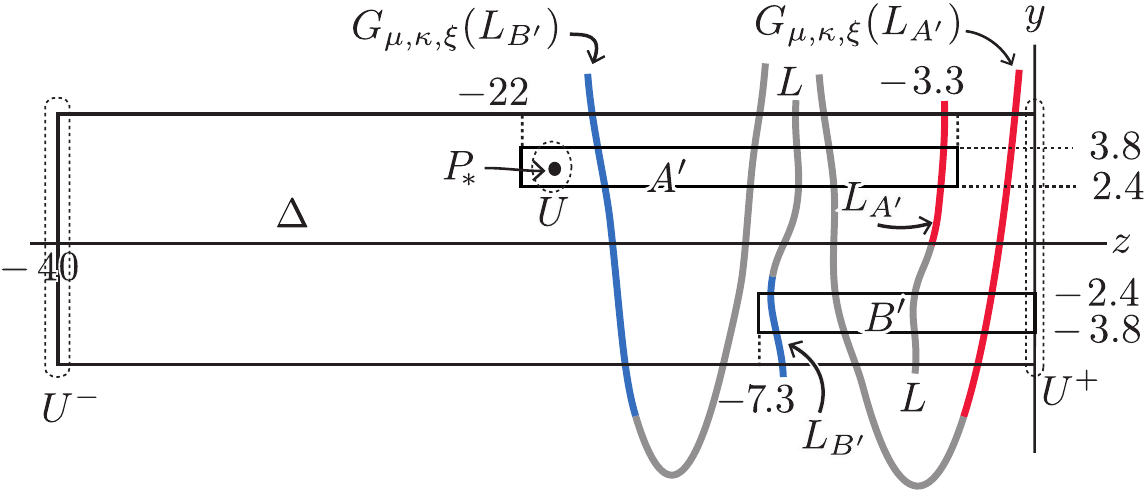}}
\caption{Vertical segments $L_{A'}$, $L_{B'}$ and their images}
\label{fig3}
\end{figure}

We define the notion of vertical curves 
through $A'$ (resp.\! $B'$) in the natural way.
That is, a curve $\sigma$ in $A'$ (resp.\ $B'$)
is called vertical 
if it is tangent to $\mathcal{C}^{uu}$
and connects $A'\cap \{ y = 3.8\}$ and $A'\cap \{ y = 2.4\}$
(resp.\! $B'\cap \{ y =  -3.8\}$ and $B'\cap \{ y = -2.4\}$).

Then, one can observe the following (see Figure \ref{fig3}):
For every  vertical curve $L$ 
through $\Delta$ to the right of $W^{s}_{0}$, 
\begin{itemize}
\item either $L_{A'} =L\cap A^{\prime}$ 
is a vertical curve through $A^\prime$, 
\item  or $L_{B'} = L\cap B^{\prime}$ 
is a vertical curve through $B^{\prime}$ 
(note that both cases may hold simultaneously).
\end{itemize}

In each case, we have the following:
\begin{itemize}
\item In the first case,
$G_{\mu,\kappa,\xi}(L_{A^{\prime}})$ contains a vertical curve
through $\Delta$ which is to the right of $W^{s}_{0}$
and disjoint from a small neighborhood $U^{+}$ of $\partial_{0}\Delta$
(remember that $\partial_{0} \Delta=\Delta\cap\{z=0\}$);
\item In the second case,
$G_{\mu,\kappa,\xi}(L_{B^{\prime}})$ contains
a vertical curve
through $\Delta$ which is  to the right of $W^{s}_{0}$
and disjoint from a small neighborhood $U$ of $W_{0}^{s}$.
\end{itemize}
Therefore,
(H5) holds for $G_{\mu,\kappa,\xi}$, and for
any diffeomorphism sufficiently $C^1$-close to $G_{\mu,\kappa,\xi}$.
\end{proof}

\begin{remark}\label{rem.bleint}
In the proof above, we can see that 
the image of the segment $\hat{\ell} :=\{ (0, t, 0) \mid |t| <4 \}$ under 
diffeomorphism $F$ which is sufficiently close to $G_{\mu, \kappa, \xi}$
contains a vertical segment which is to the right of $W_0^s$ of the blender.
Thus it has non-empty intersection with the stable manifold of the 
blender. This fact is used in Section~\ref{s.proof-of-last-thm}.
\end{remark}

\section{The non-transverse heterodimensional cycles}\label{s.Transitions}
In this section we 
describe the conditions satisfied by the diffeomorphism
with non-transverse heterodimensional cycles in Theorem \ref{thm.main1}. 

\subsection{Local dynamics at the saddle points}\label{ss.linearization}
Let $P$ and $Q$ be saddle periodic points
of $f\in \Diff^{r}(M)$ with $\dim(M)=3$ and $r\geqslant 2$.  
In what follows, we suppose that 
$\mathrm{per}(P)=\mathrm{per}(Q)=1$ for simplicity, 
that the points have indices $\ind(P)=2$ and $\ind(Q)=1$, and that the  
 unstable manifold $W^{u}(P)$ and the 
stable manifold $W^{s}(Q)$ have a heterodimensional tangency,
and the two invariant manifolds $W^{u}(Q)$ and $W^{s}(P)$ have a quasi-transverse intersection.
We fix small coordinate neighborhoods $U_{P}$  and $U_{Q}$ of $P$ and $Q$, respectively, and
consider transitions between these  neighborhoods
along the cycle.

For simplicity,  we assume that
$f\vert_{U_{P}}$ and $f\vert_{U_{Q}}$  are $C^r$ linearized,
i.e.,
\begin{equation}\label{eq.linearization}
f(x, y, z)=\left\{
\begin{array}{cc}
(\tilde\lambda x, \tilde\sigma y, \tilde\zeta z) & \mbox{if}\  (x, y, z)\in U_{P} \\
(\lambda x, \sigma y, \zeta z) &  \mbox{if}\   (x, y, z)\in U_{Q}
\end{array}\right.
\end{equation}
where
$$|\tilde\lambda|<1< |\tilde\sigma|<|\tilde\zeta|, \quad
 |\lambda|<|\zeta|<1<|\sigma|.
 $$
This assumption is guaranteed by the conditions on the eigenvalues of
$Df(P)$ and of $Df(Q)$, see \cite{S58,T71}.
In what follows 
we assume that 
$\lambda$ and $\tilde \zeta$ are both positive
(in the negative case it is enough to consider a pair number of iterations).

\begin{caveat*}
In what follows,  for simplicity we assume that  all eigenvalues are positive. For 
the cases where some eigenvalues are negative, 
the proofs work almost similarly with slight modifications.
\end{caveat*}

\subsection{Transitions along heteroclinic orbit}\label{subsection2-2}
In this subsection we describe the transitions from $U_{Q}$ to $U_{P}$ 
and from $U_{P}$ to $U_{Q}$
 along the heteroclinic orbits.

We consider first the transition from $U_{Q}$ to $U_{P}$ along the quasi-transverse 
orbit $X$.
In the linearizing coordinate in $U_{Q}$,  
up to multiplication by some constant along the  the $y$-direction,
we can take  the quasi-transverse intersection 
between $W^{u}(Q)$ and $W^{s}(P)$ with 
$X=(0, 1, 0)\in U_{Q}$.
Similarly, there exists a positive integer $N_{1}$ such that
$\tilde{X}=f^{N_{1}}(X)=(1,0,0)\in W_{\loc}^{s}(P)$, see Figure \ref{fig4}.
\begin{figure}[hbt]
\centering
\scalebox{0.55}{\includegraphics[clip]{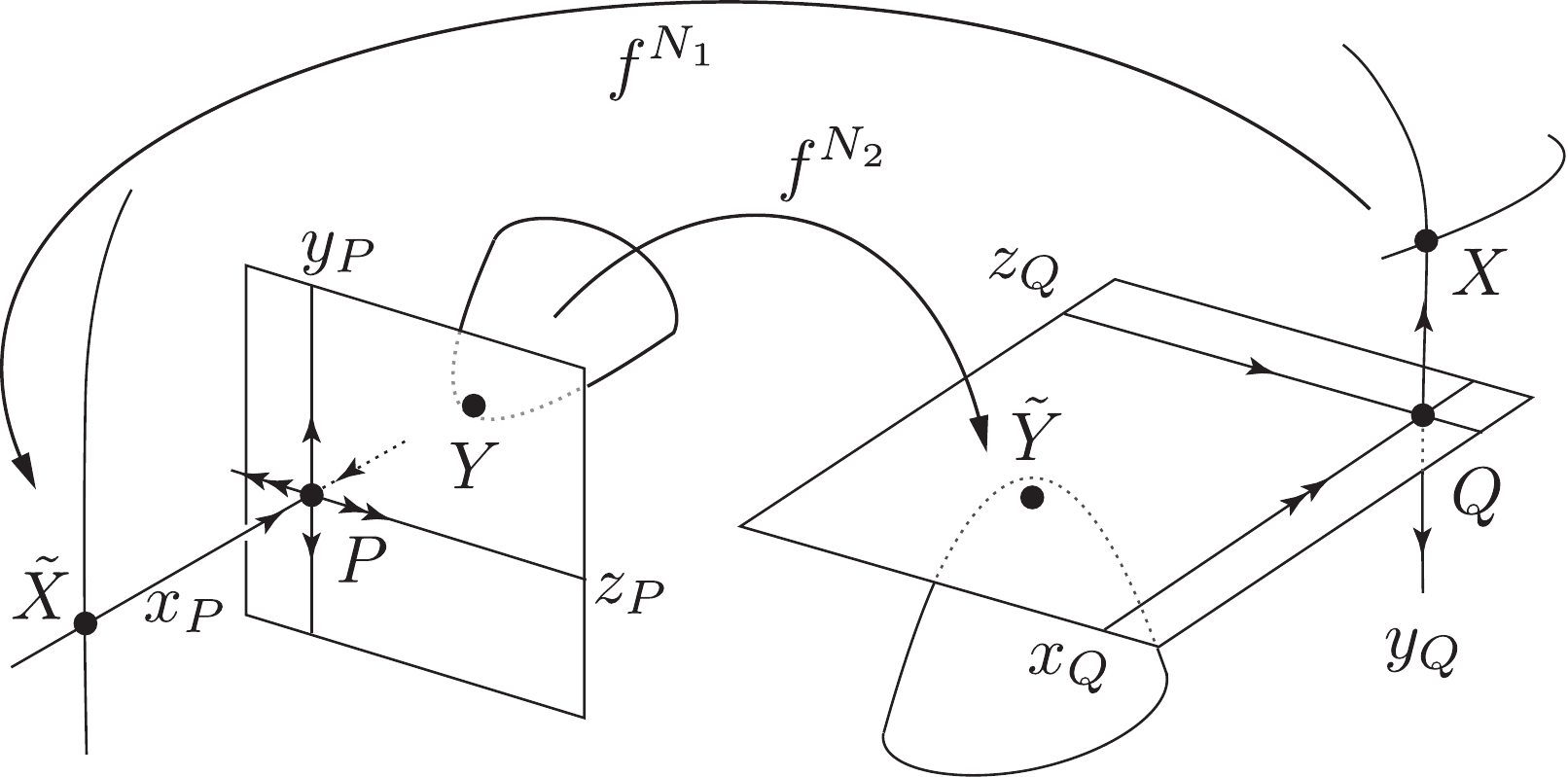}}
\caption{Non-transverse heterodimensional cycle and its transitions
}
\label{fig4}
\end{figure}

The transition of $f^{N_{1}}$ from a small neighborhood of $X$ to that of $\tilde{X}$
is expressed   as
\begin{equation}\label{eq.transition2}
f^{N_{1}}: \left(
 \begin{array}{c}
 x \\
1+y\\
z
 \end{array}
 \right)
 \longmapsto
 \left(\begin{array}{c}
1+\alpha_{1} x+\alpha_{2} y+\alpha_{3} z+\tilde{H}_{1}(x,y,z) \\
\beta_{1} x+\beta_{2} y+\beta_{3} z  +\tilde{H}_{2}(x,y,z) \\
\gamma_{1} x+\gamma_{2} y+\gamma_{3} z   + \tilde{H}_{3}(x,y,z)
 \end{array}
 \right),
\end{equation}
where each $\tilde{H}_{i}$ is the higher order term
 satisfying the following conditions: for every $i\in\{1,2,3\}$,
 $$\tilde{H}_{i}(\mathbf{0})=0; \
(\partial / \partial x)  \tilde{H}_{i}(\mathbf{0})
=(\partial / \partial y)  \tilde{H}_{i}(\mathbf{0})
=(\partial / \partial z)  \tilde{H}_{i}(\mathbf{0})
 =0.$$
 We assume that 
 \begin{equation}\label{tech-assumption1a}
\beta _{3}=\gamma _{2}=\gamma _{3}=0.
\end{equation}
Note that since $f$  is a diffeomorphism,
\begin{equation}\label{tech-assumption1b}
\beta_{2}\gamma_{1}\neq 0.
\end{equation}

Next, we consider the  transition from  $U_{P}$ to $U_{Q}$
along the orbit of heterodimensional tangency.
Let $Y\in U_{P}$ denote a point of tangency  between 
$W_{\loc}^{u}(P)$ and $W^{s}(Q)$.
There is a positive integer $N_{2}$ such that
$\tilde{Y}=f^{N_{2}}(Y)$ is  contained in $W^{s}_{\loc}(Q)$.
By some linear coordinate changes in $U_{P}$ and in $U_{Q}$,
one may set  $Y=(0, 1, 1)\in U_{P}$ and
$f^{N_{2}}(Y)=(1, 0, 1)\in U_{Q}$
respectively.
Note that this coordinate change can be done
independently of the
previous one involving $X$ and $X'$. 
Hereafter,
these new coordinates are both denoted  by $(x, y, z)$.

Since the tangency  is nondegenerate,
the transition of $f^{N_{2}}$ from a small neighborhood of $Y$ to that of $\tilde{Y}$
is expressed, by taking Taylor expansion,  as
\begin{equation}\label{eq.transition1}
f^{N_{2}}: \left(
 \begin{array}{c}
 x \\
1+y \\
1+z
 \end{array}
 \right)
 \longmapsto
 \left(\begin{array}{c}
1+a_{1} x+a_{2} y+a_{3} z+H_{1}(x,y,z) \\
b_{1} x+b_{2} y^{2}+ b_{3} z^{2} +b_{4} y z +H_{2}(x,y,z) \\
1+c_{1} x+c_{2} y+c_{3} z+H_{3}(x,y,z)
 \end{array}
 \right),
\end{equation}
where  every $H_{i}$ is  the higher order term  with the following conditions: for each
$i\in\{1,2,3\}$,
\begin{align*}
& H_{i}(\mathbf{0})=0;\
(\partial / \partial x)  H_{i}(\mathbf{0})
=(\partial / \partial y)  H_{i}(\mathbf{0})
=(\partial / \partial z)  H_{i}(\mathbf{0})
=0.
\end{align*}
We assume that 
 \begin{equation}\label{tech-assumption2}
  (\partial^{2} / \partial y^{2})  H_{2}(\mathbf{0})
 = (\partial^{2} / \partial z^{2})  H_{2}(\mathbf{0})
 = (\partial^{2} / \partial y \partial  z)  H_{2}(\mathbf{0})=0,\quad 
c_{3}=0.
\end{equation}
Furthermore, we assume that 
\begin{equation}\label{non-zero}
b_{2}b_{3}\neq 0,\quad 
\gamma_{1} a_{3}>0.
\end{equation}

\subsection{On the range of eigenvalues}\label{ss.eieva}

Theorem~\ref{thm.main1} assumes some conditions on the eigenvalues. 
In this section, we discuss the non-emptiness of the set of the 
numbers which satisfy these conditions.

We are interested in the existence of $6$-ple of numbers
$$ 0<\tilde{\lambda}< 1<\tilde{\sigma}< \tilde{\zeta},
\qquad
 0<\lambda< \zeta<1<\sigma
$$
satisfying the following conditions:
\begin{align}
& 0< (\tilde{\sigma}\tilde{\zeta})^{k} \sigma\zeta^{2}  <1, \label{ichi} \\
&0< (\tilde{\sigma}^{-3}\tilde{\zeta})^{k} \sigma^{-1}  <1, \label{ni} \\
& 0< (\tilde{\lambda}\tilde{\sigma})^{k} \sigma  <1, \label{san}
\end{align}
where $k=\dfrac{ \log \lambda^{-1}}{\log \tilde\zeta }$.

We prove the following:
\begin{lemma}\label{lem.Katsutoshi}
\begin{itemize}
\item Let   
$\mathcal{P} \subset \mathbb{R}^6$ 
denote the set of points 
$(\tilde{\lambda}, \tilde{\sigma}, \tilde{\zeta},  \lambda,  \zeta, \sigma)$
which satisfy the conditions above.
Then, $\mathcal{P}$ is a non-empty open set of $\mathbb{R}^6$.
\item On $\mathcal{P}$, the value $\log \lambda / \log \zeta$ ranges over $(1, 3/2)$.
\end{itemize}
\end{lemma}
Note that the second item implies that the range of $\alpha$
in Theorem~\ref{thm.main3} is $(0, 1/2)$.

\begin{proof}
The openness of $\mathcal{P}$ is clear. Thus we concentrate on the non-emptiness.
There are three inequalities (\ref{ichi}), (\ref{ni}) and (\ref{san}). 
First, we restrict our attention to (\ref{ichi}), (\ref{ni}).
Indeed, if we can prove the non-emptiness of numbers satisfying (\ref{ichi}) and (\ref{ni}), 
then, by taking sufficiently small $\tilde{\lambda}>0$, we can prove the non-emptiness 
of the parameter satisfying (\ref{san}).

The inequality (\ref{ichi}) is equivalent to the following:
\[
\log \sigma < 
\log \lambda \left( \frac{\log \tilde{\sigma}}{\log \tilde{\zeta}} + 1 - 2 \,\frac{\log \zeta}{\log \lambda} \right) =: R.
\]
Similarly, by a direct calculation we see that (\ref{ni}) is equivalent to 
\[
\log \sigma > - \log \lambda\left( 1 -3 \, \frac{\log \tilde{\sigma}}{\log \tilde{\zeta}}\right) =:L.
\]
Then, consider the two terms $L$ and $R$.  The non-emptiness of $\mathcal{P}$ 
is equivalent to the following two inequalities:
\begin{eqnarray}\label{eq.iikae}
L < R \quad  \mbox{and} \quad  R>0.
\end{eqnarray}
Indeed, if (\ref{eq.iikae}) holds, then take $\sigma$ so that $\log \sigma > L$ 
and $\log \sigma <R$ hold (note that the variable $\sigma$ does not appear in $L$ and $R$  and 
the only restriction on $\sigma$ is $\sigma > 1$, that is, $\log \sigma >0$).

Put
\[
S :=\frac{\log \tilde{\sigma}}{\log \tilde{\zeta}}, \quad 
T := \frac{\log \zeta}{\log \lambda}.
\]
Note that the conditions $ 1 <\tilde{\sigma} < \tilde{\zeta}$ and 
$1 > \zeta > \lambda$ implies that we have 
$$0 < S, T <1.$$

The inequalities $L < R$ and $R>0$ are respectively equivalent to the following:
\[
 T >-S+1 \quad \mbox{and} \quad T > (1/2)(S+1)
\]
By Figure \ref{fig5} below,
we know that 
the set of $(S, T)$ which satisfies these conditions is non-empty.

\begin{figure}[hbt]
\centering
\scalebox{0.8}{\includegraphics[clip]{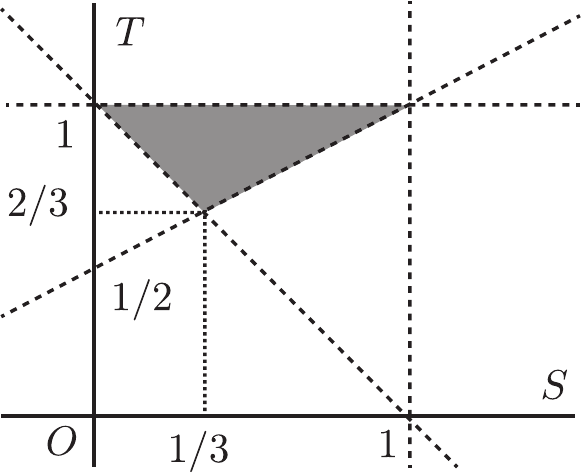}}
\caption{The domain of $(S, T)$ 
}
\label{fig5}
\end{figure}

Now we prove the non-emptiness of $\mathcal{P}$. First, fix some $(S, T)$ 
from the region of Figure \ref{fig5}.
Then, take $(\tilde{\sigma}, \tilde{\zeta})$ and  
$(\zeta, \lambda)$ which correspond to the values of $(S, T)$. 
Finally, take $\sigma$ and $\tilde{\lambda}$. This gives us the desired 
6-ple of numbers.

Finally, we see the range of 
$\frac{\log \lambda}{\log \zeta} = 1/T$. 
By the figure above, the range of $T$ is $2/3< T < 1$. Thus 
the range of $1/T$ is $(1, 3/2)$, which completes the proof.
\end{proof}

%

\section{The six-parameter family $\{f_{\mu,\nu}\}$ }\label{ss.Local perturbations}
In this section we construct  a
six-parameter family $\{f_{\mu,\nu} \} \subset \Diff^{r}(M)$, $\mu, \nu\in [-\epsilon, \epsilon]^{3}$, 
with $f_{\bf 0,0}=f$ in Theorem \ref{thm.main1}. This family 
is obtained by local perturbations
near the quasi-transverse intersection $\tilde X$
and the heterodimensional tangency $\tilde Y$.
To define the local perturbations,
we use a smooth bump function
\begin{equation}\label{bump function}
B(x,y,z) = b(x)b(y)b(z),
\end{equation}
where $b$ is a  $C^{r}$ function  on $\mathbb{R}$
satisfying
$$
\begin{cases}
b(x)= 0\ & \mbox{if}\  2\rho \leqslant \vert x\vert; \\
0< b(x) <1\ & \mbox{if}\  \rho<|x|<2\rho;\\
b(x)= 1\  &  \mbox{if}\  \vert x\vert  \leqslant\rho,
\end{cases}
$$
where $\rho>0$ is some small number. 
Let
$U_{\tilde X}$ and $U_{\tilde Y}$
be  $2\rho$-neighborhoods
of $\tilde X=(1,0,0)$
and of  $\tilde Y=(1,0,1)$
which  satisfy
$U_{\tilde X}\subset U_{P}$, $P\not\in \Cl(U_{\tilde X})$ and
$U_{\tilde Y}\subset U_{Q}$, $Q\not\in \Cl(U_{\tilde Y})$,
where $\Cl (\, \cdot \,)$ means the closure of the corresponding set.

The number $\rho$ is taken so small such that 
\[
f(U_{\tilde{X}}) \cap U_{\tilde{X}}= \emptyset  \,\,\, \mbox{and} \,\,  \,
f(U_{\tilde{Y}}) \cap U_{\tilde{Y}} = \emptyset 
\]
hold.

For $\mu=(\mu_{1},\mu_{2}, \mu_{3}), \nu=(\nu_{1},\nu_{2}, \nu_{3})\in \mathbb{R}^{3}$,
let $\{t_{\mu,\nu}\}$ be a family of maps
satisfying the following:
\begin{itemize}
\item if  $(1+x, y, 1+z)\in U_{\tilde Y}$,
$$
t_{\mu,\nu}(1+x, y, 1+z) =(1+x, y, 1+z)+
B(x,y,z) (\mu_{1},\mu_{2}, \mu_{3}),
$$
\item if $(1+x, y, z)\in U_{\tilde X}$,
$$
t_{\mu,\nu}(1+x, y,  z) =(1+x,  y,   z)+
B(x,y,z)  (\nu_{1},  \nu_{2}, \nu_{3}),
$$
\item in the complement of $U_{\tilde X}\cup U_{\tilde Y}$,
$t_{\mu,\nu}$ is the identity.
\end{itemize}
Then this family is in $\Diff^{r}(M)$ for small $\mu, \nu\in [-\epsilon, \epsilon]^{3}$ where $ \epsilon>0$  is some small number. 

We now define
\begin{equation}\label{def-two-parameter-family}
f_{\mu,\nu}=t_{\mu,\nu}\circ f,
\end{equation}
which has the following properties.
 \begin{itemize}
\item Since $P\not\in \Cl(U_{\tilde X})$  and $Q\not\in \Cl(U_{\tilde Y})$,
 $f_{\mu,\nu}$ has   saddle periodic points $P_{\mu,\nu}=P$ and 
$Q_{\mu,\nu}=Q$.
 Moreover, near $P$ and $Q$,
 $f_{\mu,\nu}$ has the same  forms as in (\ref{eq.linearization}).

\item
For any $(x,1+y, z)$ sufficiently near $X=(0, 1, 0)$,
\begin{equation}\label{eq.transition2'}
f^{N_{1}}_{\mu, \nu}(x,1+y,z)=f^{N_{1}}(x,1+y,z)+ (\nu_{1},  \nu_{2}, \nu_{3}),
\end{equation}
which  has the same form 
as  (\ref{eq.transition2}) when $\nu=\mathbf{0}$.
The parameters $\nu_{2}$ and $\nu_{3}$ control the unfolding of the 
quasi-transverse intersection $\tilde{X}$. See Figure \ref{fig6}-(a).

 \item
 For any $(x, y+1, z+1)$ sufficiently near $Y=(0, 1, 1)$,
 \begin{equation}\label{eq.transition1'}
f^{N_{2}}_{\mu, \nu}(x,1+y,1+z)=f^{N_{2}}(x,1+y,1+z)+ (\mu_{1},  \mu_{2}, \mu_{3}),
\end{equation}
which  has the same form as  (\ref{eq.transition1}) when $\mu=\mathbf{0}$.
The parameter $\mu_{2}$ controls the unfolding of the
 heterodimensional tangency  at $\tilde{Y}$.   See Figure \ref{fig6}-(b).
 \end{itemize}
\begin{figure}[hbt]
\centering
\scalebox{0.85}{\includegraphics[clip]{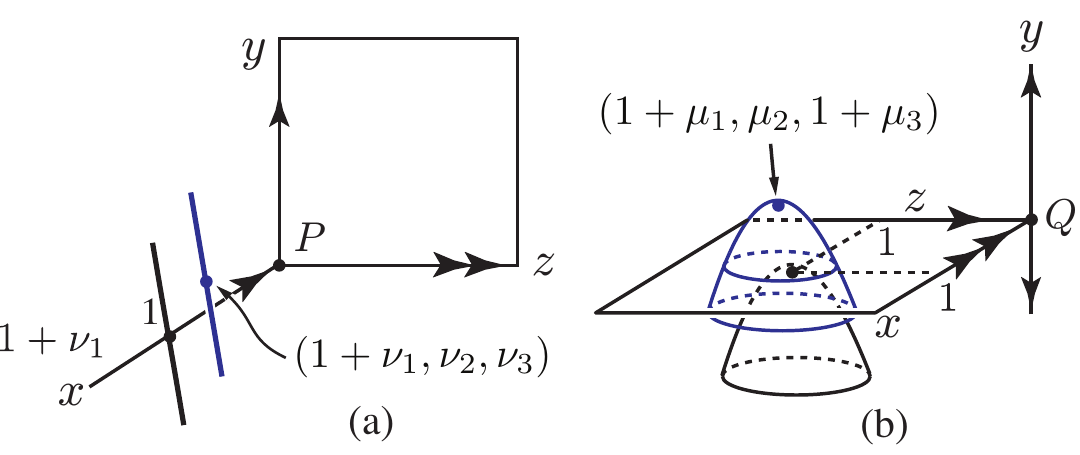}}
\caption{(a) Quasi-transverse intersection  (b) heterodimensional tangency
}
\label{fig6}
\end{figure}

\section{Renormalizations: Proof of Theorem \ref{thm.main1}}\label{s.Proof of Theorem1}
%

The renormalization scheme involves maps of the form 
$ f^{N_{2}+m+N_{1}+n}_{\mu, \nu}$
where $N_{1}$ and $N_{2}$ are (fixed) transition times, 
and $n$ and $m$ are the permanence times in $U_{Q}$ and $U_{P}$, respectively.
We pick up pair of sojourn times to get return map with neutral behavior.

\subsection{For selecting the sojourn times}\label{ss.Essential results}
We first prove an auxiliary result which enables us to get the 
``neutral dynamics''. For that we state 
an ``irrationality condition'' of real numbers (Claim \ref{cla.residual set}).
This condition is used to choose a convenient combination of product of 
the  strong unstable eigenvalue $\tilde{\zeta}$ of $P$ and
the strong stable one  $\lambda$ of $Q$  in (\ref{eq.linearization}).
With this condition we pick up a sequence of renormalizations
converging to a central unstable H\'enon-like family. 

We put
$$\mathcal{Z}:=\{(\tilde \zeta,\lambda) \in \mathbb{R}^2\  \vert\ 0< \lambda <1<\tilde{\zeta} \}.$$

\begin{lemma}\label{lem.Neutrality}
There is a residual subset $\mathcal{R}$ of $\mathcal{Z}$ such that 
for every
$(\tilde{\zeta}, \lambda)\in  \mathcal{R}$ the following holds:
For any $\varepsilon>0$,  $N_{0}>0$
and $\xi >0 $ satisfying $\varepsilon  /\xi <1$,
there exist  integers $m, n>N_{0}$ such that the following holds: 
\[
\left|  \gamma_{1}  a_{3} \lambda^{n} \tilde{\zeta}^{m} - \xi \right| < \varepsilon, \quad 
 \vert m-nk -  \tilde{k} \vert<1, 
\] 
where  $\gamma_{1}$ and  $a_{3}$ are the constants given
 in (\ref{eq.transition2}), (\ref{eq.transition1}) and (\ref{non-zero}),
$k={\log \lambda^{-1} }/{\log \tilde{\zeta}  }$ 
and
$\tilde{k}=\log (\gamma_{1} a_{3} \xi^{-1}) / \log  \tilde{\zeta} $.
\end{lemma}

To prove this lemma we need the following claim whose proof is postponed.
\begin{claim}\label{cla.residual set}
There exists a residual subset $ \mathcal{R}$  of $ \mathcal{Z}$ 
such that  for every   $(\tilde{\zeta}, \lambda)\in \mathcal{R}$ the following holds:
$$
\Cl  \left( \left\{
m\log \tilde{\zeta} + n\log \lambda \ \middle\vert \ m, n\in \mathbb{N}
\right\} \right)=\mathbb{R}.
$$
\end{claim}

Assuming this claim let us give a proof of the lemma above.
\begin{proof}[Proof of Lemma \ref{lem.Neutrality}]
Fix $\varepsilon>0$,  $N_{0}>0$ and $\xi > 0$. 
We take $\varepsilon_0 >0$ such that 
$ e^{\varepsilon_0} -1 < \varepsilon / \xi$.
By the condition $\varepsilon /\xi <1$, we have  $\varepsilon_0 <1$.
From Claim~\ref{cla.residual set} and (\ref{non-zero}),
for any  $\varepsilon_0>0$,  $N_{0}>0$ and $\xi >0 $,
there are integers $m, n>N_{0}$ such that
\begin{eqnarray} 
\left\vert
\log (\gamma_{1}  a_{3}  \xi^{-1})
+ m\log\tilde{\zeta}+n\log\lambda\right\vert
<\varepsilon_0.   \label{conve}
\end{eqnarray}
For such $m$ and $n$, we have 
$\vert m-nk -  \tilde{k} \vert
<\varepsilon_0/\log\tilde{\zeta} <1$.

By taking the exponential of (\ref{conve}), we have 
$$
\xi e^{-\varepsilon_0} < \gamma_{1}  a_{3} \lambda^{n}  \tilde{\zeta}^{m} < \xi e^{\varepsilon_0}.
$$
Then, by subtracting $\xi$ from each side, 
we get $ \left| \gamma_{1}  a_{3}  \lambda^{n} \tilde{\zeta}^{m} - \xi \right| < \varepsilon$.
\end{proof}

\begin{proof}[Proof of Claim \ref{cla.residual set}]
Let $\{ r_{i}\}_{i\in \mathbb{N}}$ be a sequence which is dense in $\mathbb{R}$.
For every $r_{i}$ and $j\in \mathbb{N}$, let us consider 
$U_{i, j} \subset \mathcal{Z}$
defined as follows:
$$
U_{i,j}:=\left\{ (\tilde{\zeta}, \lambda)\ \middle\vert\
\exists\ m, n\in \mathbb{N}\text{ s.t. } 
| m\log\tilde{\zeta}+n\log\lambda-r_{i} |<1/j
\right\}.
$$
The sets
$U_{i, j}$ are clearly open in $\mathcal{Z}$. 
for every positive integers 
$i$ and $j$. We claim the
density of $U_{i, j}$.
Then
 $\mathcal{R}:= \cap_{i, j} U_{i, j}$ 
 is the desired residual  set (by Baire's category theorem).
Consider  the (open) set  
$$V_{i,j}(m,n) :=\left\{(X,Y)
=(\log\tilde{\zeta}, \log\lambda) \in \mathbb{R}^2
\ \middle\vert\ | mX+nY-r_{i} |<1/j
\right\}.$$
Observe that, in $XY$-plane, $V_{i,j}(m,n)$ is an open strip containing the
(open) segment $((r_{i}-1/j)/n, (r_{i}+1/j)/n)$ on the $Y$-axis,
going downward right in the fourth quadrant 
$(0, +\infty)\times (-\infty, 0)$ with slope $-m/n$.
Then, by the density of rational numbers $\{ m/n \, ; \, m, n \in \mathbb{Z}_{>0} \}$ and taking $n$ large,
we get 
$$
\Cl \left( \bigcup_{m, n >0}  
 V_{i,j}(m,n)\ \right)=(0, +\infty)\times (-\infty, 0).
$$
Then the pullback of $\cup_{m, n} V_{i, j}(m, n)$ to $\mathcal{Z}$ 
under the map $(x,y)\mapsto (\log |x|, \log|y|)$
coincides with $U_{i, j}$, which shows the density of $U_{i, j}$.
\end{proof}

\subsection{Renormalizations near tangencies}\label{ss.reno}
Before going our construction 
recall the conditions
(\ref{eigenvalue condition 1})--(\ref{eigenvalue condition 2}) and 
 (\ref{eq.transition2})--(\ref{non-zero}).
In what follows 
we assume that $(\lambda, \tilde{\zeta})$ belongs to the residual subset $ \mathcal{R}$ 
of Lemma \ref{lem.Neutrality}.

Our renormalization scheme  consists of 
a sequence of coordinate changes $\Psi_{m,n}$,
reparametrizations of 
$\mu_{m,n}$ and 
parameters $\nu_{m,n}$ depending on given integers $m, n>0$, 
which satisfy the following conditions.

\begin{itemize}
\item
The coordinate change 
$\Psi_{m, n} : K \to U_Q$,
(where $K$ is a compact neighborhood of the origin)
$(x, y, z)=\Psi_{m,n}( \bar{x} , \bar{y}, \bar{z})$ is defined by
\begin{equation}\label{coordinate change}
\Psi_{m,n}( \bar{x} , \bar{y}, \bar{z}) :=
(\sigma^{-n}  \tilde{\sigma}^{-m} \bar{x}  +1,\
\sigma^{-2n} \tilde{\sigma}^{-2m} \bar{y} +\sigma^{-n},\
\sigma^{-n}\tilde{\sigma}^{-m} \bar{z} +1).
\end{equation}
Note that 
$\Psi_{m, n}(K)$ 
converges to the point of heterodimensional tangency $\tilde{Y}=(1, 0, 1)\in U_{Q}$  as $m,n\rightarrow \infty$.
This means that for any $K$ if $m, n$ are sufficiently large, 
we can define $\Psi_{m, n}$.

\item
The reparametrization 
$\mu_{m,n}:I\to \mathbb{R}^{3}$, 
$\mu=\mu_{m,n}(\bar\mu)$, 
is defined by 
\begin{equation}\label{mu-reparametrization}
\mu_{m,n}(\bar\mu) :=
(
-\tilde{\lambda}^{m} a_{1},\
\sigma^{-2n}\tilde{\sigma}^{-2m} \bar{\mu}+\sigma^{-n}-\tilde{\lambda}^{m} b_{1}, \
-\tilde{\lambda}^{m} c_{1}
),
\end{equation}
which converges to $\mathbf{0}$ as $m, n\rightarrow \infty$ 
where $I$ is a (fixed) closed interval. 

\item The sequence of parameter values
$\{ \nu_{m,n} \}$
is
given as
\begin{equation}\label{nu-reparametrization}
\nu_{m,n}
:=(
-\lambda^{n} \alpha_{1}-\zeta^{n} \alpha_{3}, \
\tilde{\sigma}^{-m}-\lambda^{n} \beta_{1},\
\tilde{\zeta}^{-m}-\lambda^{n} \gamma_{1}),
\end{equation}
which converges to $\mathbf{0}$ as $m, n\rightarrow \infty$.
\end{itemize}

Using  the renormalization  $(\Psi_{m,n}, \mu_{m,n}, \nu_{m,n})$,
the \emph{return map} by iterations of  $f_{\mu,\nu}$
near the heterodimensional tangency  $\tilde{Y}$
is defined by 
\begin{equation}\label{return map}
F_{m,n}(\bar{x}, \bar{y}, \bar{z}):=
(\Psi_{m,n}^{-1}|_{\Psi_{m, n}(K)})
\circ f^{N_{2}+m+N_{1}+n}_{\mu_{m,n}(\bar\mu),\nu_{m,n}}
\circ \Psi_{m,n}(\bar{x}, \bar{y}, \bar{z}),
\end{equation}
where $N_{1}$ and $N_{2}$ are the  constants in  (\ref{eq.transition2}) and (\ref{eq.transition1}), 
which are independent of $m,n$.
Note that  the domain of $F_{m, n}$ is $K$ and 
$K$ can be chosen arbitrarily large
by letting $m, n$ large.

\subsection{Proof of Theorem \ref{thm.main1}}\label{ss.Proof of Theorem 1.1}
In this subsection, we prove Theorem \ref{thm.main1}.
For that we fix 
 $\xi >0$ and 
select a sequence $(m_k, n_k) \subset \mathbb{N}^2$
such that $\gamma_{1}  a_{3}  \lambda^{n_k} \tilde{\zeta}^{m_k}$
converges to $\xi$ (recall Lemma \ref{lem.Neutrality}). 
Then
for those $\{F_{m_k, n_k} \}$ we get the convergence to the center unstable H\'enon-like family.

 \medskip
 
\noindent
\emph{Proof of (1)}. 
The claim  is obtained immediately
from (\ref{coordinate change})--(\ref{nu-reparametrization})
under the conditions for eigenvalues in  (\ref{eigenvalue condition 1}).

 \medskip
 
\noindent
\emph{Proof of (2)}.  
The proof
is done by calculating entries in the formula (\ref{return map})  using
 (\ref{eigenvalue condition 1}), (\ref{eigenvalue condition 2}),
 (\ref{eq.transition2})--(\ref{non-zero}) and
 $(\lambda, \tilde{\zeta})\in \mathcal{R}$ of Lemma \ref{lem.Neutrality}. 
 In the proof, we omit the subscript $k$ for 
 simplicity. Thus, for example, we write $F_{m, n}$ 
 in the sense of $F_{m_k, n_k}$.

Let us provide  step-by-step calculations to obtain entries in the formula of return map $F_{m,n}$.
By the coordinate change $\Psi_{m,n}$ of  (\ref{coordinate change}),
each $(\bar{x}, \bar{y}, \bar{z})\in \mathbb{R}^{3}$
is mapped to
$$\mathbf{x}_{0}:=
(\sigma^{-n}  \tilde{\sigma}^{-m} \bar{x}  +1,\
\sigma^{-2n} \tilde{\sigma}^{-2m} \bar{y} +\sigma^{-n},\
\sigma^{-n}\tilde{\sigma}^{-m} \bar{z} +1).
$$
Note that if $(\bar{x}, \bar{y}, \bar{z})$ is contained in  a compact domain and $m,n$ are sufficiently large, 
 $\mathbf{x}_{0}$ is close
to  the heterodimensional tangency  $\tilde Y=(1,0,1)$.

First, after $n$ iterations of $f_{\mu_{m,n}(\bar{\mu}),\nu_{m,n}}$, in other words, 
 the linear transformation (\ref{eq.linearization}),
 $\mathbf{x}_{0}$ moves  to  $f^{n}_{\mu_{m,n}(\bar{\mu}),\nu_{m,n}}(\mathbf{x}_{0})=(x_{n}, y_{n}+1, z_{n})$ where
$$
(x_{n}, y_{n}+1, z_{n})=
( \lambda^{n} \sigma^{-n}  \tilde{\sigma}^{-m} \bar{x}  + \lambda^{n},\
\sigma^{-n} \tilde{\sigma}^{-2m} \bar{y} +1,\
\zeta^{n}\sigma^{-n}\tilde{\sigma}^{-m} \bar{z} +\zeta^{n}).
$$
We write  $\tilde{\mathbf{x}}_{n}= (x_{n}, y_{n}, z_{n})$.
These entries  can be described by Landau notation as 
\begin{equation}\label{Landau-notation1}
x_{n}=O(\lambda^{n}),\
y_{n}=O(\sigma^{-n}\tilde{\sigma}^{-2m}),\
z_{n}=O(\zeta^{n}).
\end{equation} 
Note that this also implies that the point
$(x_{n}, y_{n}+1, z_{n})$ converges to $(0, 1, 0)$ when $n$ 
tends to $+\infty$. 
This guarantees that the points in 
$\Psi_{m, n}(\Delta)$ stay in $U_Q$ and then leave it by the transition map.

Thus, 
we apply the transition $f^{N_{1}}_{\mu_{m,n}(\bar{\mu}),\nu_{m,n}}$ defined  by
 (\ref{eq.transition2'}) 
  (where the parameter $\mu_{m,n}(\bar{\mu})$ is not yet involved).
Recalling the choice of $\nu$ in  (\ref{nu-reparametrization}) and the conditions 
$\beta _{3}=\gamma _{2}=\gamma _{3}=0$ in (\ref{tech-assumption1a}) 
we get 
the following formula  in the local coordinate of $U_{P}$,
\begin{align*}
& f^{N_{1}+n}_{\mu_{m,n}(\bar{\mu}),\nu_{m,n}}(\mathbf{x}_{0})=
\bigr( 1+\lambda^n \sigma^{-n}  \tilde{\sigma }^{-m} \alpha _1 \bar{x}
+\sigma^{-n}  \tilde{\sigma }^{-2 m} \alpha_2\bar{y}
+\zeta^n \sigma^{-n}  \tilde{\sigma }^{-m} \alpha _3 \bar{z}
+\tilde{H}_{1}(\tilde{\mathbf{x}}_{n}), \\
&\tilde{\sigma}^{-m}
+\lambda^n \sigma ^{-n} \tilde{\sigma}^{-m} \beta _1\bar{x}
+\sigma^{-n}  \tilde{\sigma}^{-2 m} \beta _2\bar{y}
+\tilde{H}_2 (\tilde{\mathbf{x}}_{n}),
\tilde{\zeta }^{-m}
+\lambda^n \sigma^{-n}  \tilde{\sigma }^{-m} \gamma _1 \bar{x}
+\tilde{H}_3(\tilde{\mathbf{x}}_{n})
 \bigr).
\end{align*} 
Note that from (\ref{eq.transition2}),
for each $i=1,2,3$, 
$\tilde{H}_i$ starts from a quadratic term.
Since $0<\lambda<\zeta$, 
the dominant terms in  $\tilde{H}_i$ are 
 $y_{n}=O(\sigma^{-n}\tilde{\sigma}^{-2m})$ and $z_{n}=O(\zeta^{n})$ in (\ref{Landau-notation1}).
 In fact,  
\begin{equation}\label{Landau-notation2}
\tilde{H}_i(\tilde{\mathbf{x}}_{n})=
O(\sigma^{-2n}\tilde{\sigma}^{-4m})+O(\zeta^{2n}).
\end{equation}  

By the next $m$ iterations of the linear transformation by (\ref{eq.linearization}) in $U_{P}$,
we have 
$f^{m+N_{1}+n}_{\mu,\nu_{m,n} }(\mathbf{x}_{0})=(x_{m}, 1+y_{m}, 1+z_{m})$
near $Y=(0,1,1)$ where
\begin{align*}
 & x_{m}
 =\tilde{\lambda }^m
+\lambda ^n \tilde{\lambda }^m \sigma ^{-n} \tilde{\sigma}^{-m} \alpha_1\bar{x}
+\tilde{\lambda}^m \sigma ^{-n}  \tilde{\sigma}^{-2 m} \alpha_2\bar{y}+\tilde{\lambda}^m \sigma^{-n} \tilde{\sigma }^{-m} \zeta^n\alpha _3\bar{z}
+\tilde{\lambda}^m \tilde{H}_1(\tilde{\mathbf{x}}_{n}),\\
 & 1+y_{m}
=1+\lambda ^n \sigma ^{-n}\beta_1 \bar{x}
+\sigma ^{-n} \tilde{\sigma}^{-m} \beta_2\bar{y}
+\tilde{\sigma}^m \tilde{H}_2(\tilde{\mathbf{x}}_{n}),\\
 & 1+z_{m}
 =1
+ \lambda ^n \sigma ^{-n} \tilde{\sigma }^{-m}  \tilde{\zeta }^m \gamma _1 \bar{x}
+\tilde{\zeta }^m \tilde{H}_3(\tilde{\mathbf{x}}_{n}).
\end{align*}
Put  $\mathbf{x}_{m}:=(x_{m}, y_{m}, z_{m})$.
From (\ref{eigenvalue condition 2}), (\ref{Landau-notation1}) and (\ref{Landau-notation2}),
we have the following:
\begin{equation}\label{Landau-notation3}
x_{m}= O(\tilde{\lambda}^{m}),\ 
y_{m}= O(\sigma^{-2n}\tilde{\sigma}^{-3m})+ O(\tilde{\sigma}^{m}\zeta^{2n}),\ 
z_{m}= O(\sigma^{-n}\tilde{\sigma}^{-m}).
\end{equation}
We can see that these three numbers converge to zero as $m, n \to +\infty$.
Indeed, the convergence of $x_m$ and $z_m$ are easy. 
The convergence of $y_m$ comes from the 
first condition of the eigenvalue condition (\ref{eigenvalue condition 2}),
This means that after $m$-times iteration, 
the point leaves $U_P$ by the transition map. 

By the transition $ f_{\mu_{m,n}(\bar{\mu}),\nu_{m,n}}^{N_{2}}$ 
(which does not depend on $\nu_{m,n}$), 
we have in the local coordinate 
 $(\hat{x}, \hat{y}, \hat{z})=
 f^{N_{2}+m+N_{1}+n}_{\mu_{m,n}(\bar{\mu}),\nu_{m,n}}(\mathbf{x}_{0})$ 
in  $U_{Q}$, where
\begin{flalign*}
\hat{x}
= &
 1
 +(\lambda^n   \tilde{\lambda }^m  \sigma ^{-n}\tilde{\sigma }^{-m} a_1 \alpha _1
 +\lambda ^n \sigma ^{-n}  a_2 \beta _1
 +\lambda ^n \sigma ^{-n}  \tilde{\sigma }^{-m}  \tilde{\zeta }^m a_3 \gamma _1) \bar{x} &\\
 &
 + (\tilde{\lambda }^m \sigma ^{-n}   \tilde{\sigma }^{-2m} a_1 \alpha _2
+\sigma ^{-n}  \tilde{\sigma }^{-m} a_2 \beta _2) \bar{y}
+ \tilde{\lambda }^m \sigma ^{-n} \tilde{\sigma }^{-m} \zeta ^n  a_1 \alpha _3  \bar{z} &\\
&
+\tilde{\lambda }^m a_1 \tilde{H}_1(\tilde{\mathbf{x}}_{n})+\tilde{\sigma }^m a_2 \tilde{H}_2(\tilde{\mathbf{x}}_{n})+\tilde{\zeta }^m a_3 \tilde{H}_3(\tilde{\mathbf{x}}_{n})+H_1(\mathbf{x}_{m}),&
\end{flalign*}
\begin{flalign*}
\hat{y}=
&
\sigma ^{-2 n} \tilde{\sigma }^{-2 m} \bar{\mu }
+ \sigma ^{-n} &\\
&
+  \lambda ^n  \tilde{\lambda }^m \sigma ^{-n}  \tilde{\sigma }^{-m} b_1 \alpha _1 \bar{x}
+ \tilde{\lambda }^m \sigma ^{-n}  \tilde{\sigma }^{-2 m} b_1 \alpha _2  \bar{y}
+  \tilde{\lambda }^m \sigma ^{-n}  \tilde{\sigma }^{-m}  \zeta ^n b_1 \alpha _3 \bar{z}&\\
&
+(\lambda ^{2 n} \sigma ^{-2 n}  \beta _1^2 b_2
+\lambda ^{2 n} \sigma ^{-2 n}   \tilde{\sigma }^{-2 m} \tilde{\zeta }^{2 m} \gamma _1^2 b_3
+\lambda ^{2 n} \sigma ^{-2 n}   \tilde{\sigma }^{-m} \tilde{\zeta }^m \beta _1 \gamma _1 b_4)\bar{x}^2 &\\
&
+\sigma ^{-2n}  \tilde{\sigma }^{-2m} \beta _2^2 b_2 \bar{y}^2
+(2 \lambda^n \sigma ^{-2 n} \tilde{\sigma }^{-m} \beta _1 \beta _2 b_2
+\lambda ^n \sigma ^{-2 n}  \tilde{\sigma }^{-2 m}  \tilde{\zeta }^m \beta _2 \gamma _1 b_4)  \bar{x} \bar{y} &\\
&
+\tilde{\lambda }^m b_1 \tilde{H}_1(\tilde{\mathbf{x}}_{n})
+(2 \lambda ^n \sigma ^{-n}  \tilde{\sigma }^m \beta _1 b_2 \bar{x}
+\lambda ^n \sigma ^{-n}  \tilde{\zeta }^m \gamma _1 b_4  \bar{x}
+2 \sigma ^{-n}  \beta _2 b_2  \bar{y}
 )\tilde{H}_2(\tilde{\mathbf{x}}_{n}) &\\
&
+(\lambda ^n \sigma ^{-n} \tilde{\zeta }^m \beta _1 b_4  \bar{x}
+2 \lambda ^n \sigma ^{-n}  \tilde{\sigma }^{-m}  \tilde{\zeta }^{2 m} \gamma _1 b_3 \bar{x}
+\sigma ^{-n}  \tilde{\sigma }^{-m}  \tilde{\zeta }^m \beta _2 b_4 \bar{y}
)  \tilde{H}_3(\tilde{\mathbf{x}}_{n}) &\\
&
+\tilde{\sigma }^{2 m} b_2 \tilde{H}_2(\tilde{\mathbf{x}}_{n})^2
+ \tilde{\sigma }^m b_4 \tilde{\zeta }^m \tilde{H}_2(\tilde{\mathbf{x}}_{n}) \tilde{H}_3(\tilde{\mathbf{x}}_{n})+\tilde{\zeta }^{2 m} b_3 \tilde{H}_3(\tilde{\mathbf{x}}_{n})^2+H_2(\mathbf{x}_{m}),&
\end{flalign*}
\begin{flalign*}
\hat{z}=
&
1
+(\lambda^n \tilde{\lambda }^m \sigma ^{-n}   \tilde{\sigma }^{-m} c_1 \alpha _1
+\lambda ^n \sigma ^{-n}  c_2 \beta _1) \bar{x} &\\
&
+(\tilde{\lambda }^m \sigma ^{-n}   \tilde{\sigma }^{-2 m} c_1 \alpha _2
+\sigma ^{-n}  \tilde{\sigma }^{-m} c_2 \beta _2 ) \bar{y}
+ \tilde{\lambda }^m \sigma^{-n}  \tilde{\sigma }^{-m} \zeta^n c_1 \alpha _3 \bar{z} &\\
&
+\tilde{\lambda }^m c_1 \tilde{H}_1(\tilde{\mathbf{x}}_{n})
+\tilde{\sigma }^m c_2 \tilde{H}_2(\tilde{\mathbf{x}}_{n})
+H_3(\mathbf{x}_{m}).&
\end{flalign*}

Finally, using the  inverse of  (\ref{coordinate change}),
we get that 
the  return map $F_{m,n}(\bar{x}, \bar{y}, \bar{z})=(\bar{x}_{1}, \bar{y}_{1}, \bar{z}_{1})$ defined by   (\ref{return map})
satisfies
\begin{equation}\label{mn-dependent return map1}
\begin{aligned}
\bar{x}_{1}
&
=
 (\lambda ^n  \tilde{\lambda }^m a_1 \alpha _1
+\lambda ^n \tilde{\sigma }^m a_2 \beta _1
+\lambda ^n  \tilde{\zeta }^m a_3 \gamma _1)\bar{x}
+
( \tilde{\lambda }^m \tilde{\sigma }^{-m} a_1 \alpha _2+ a_2 \beta _2)\bar{y}& \\
&
+
\tilde{\lambda }^m \zeta ^n   a_1 \alpha _3\bar{z}
+\tilde{\lambda }^m \sigma ^n \tilde{\sigma }^m a_1 \tilde{H}_1(\tilde{\mathbf{x}}_{n})
+\sigma ^n \tilde{\sigma }^{2 m} a_2 \tilde{H}_2(\tilde{\mathbf{x}}_{n})&\\
&
+ \sigma ^n  \tilde{\sigma }^m \tilde{\zeta }^m  a_3  \tilde{H}_3(\tilde{\mathbf{x}}_{n})
+\sigma ^n \tilde{\sigma }^m H_1(\mathbf{x}_{m}),&    
\end{aligned}
\end{equation}
\begin{equation}\label{mn-dependent return map2}
\begin{aligned}
\bar{y}_{1}
&
=\bar{\mu}
+\lambda ^n  \tilde{\lambda }^m \sigma ^n \tilde{\sigma }^m b_1 \alpha _1\bar{x}
+\tilde{\lambda }^m \sigma ^n  b_1 \alpha _2\bar{y}
+  \tilde{\lambda }^m \sigma ^n  \tilde{\sigma }^m \zeta ^n b_1 \alpha _3\bar{z}&\\
&
+(\lambda ^{2 n}  \tilde{\sigma }^{2 m} \beta _1^2 b_2
+\lambda ^{2 n}  \tilde{\zeta }^{2 m} \gamma _1^2 b_3
+\lambda ^{2 n}  \tilde{\sigma }^m \tilde{\zeta }^m  \beta _1 \gamma _1 b_4) \bar{x}^2
+ \beta _2^2 b_2 \bar{y}^2&\\
&
+(2 \lambda ^n  \tilde{\sigma }^m \beta _1 \beta _2 b_2
+\lambda ^n \tilde{\zeta }^m \beta _2 \gamma _1 b_4) \bar{x} \bar{y}
+\tilde{\lambda }^m \sigma ^{2 n}  \tilde{\sigma }^{2 m} b_1 \tilde{H}_1(\tilde{\mathbf{x}}_{n})&\\
&
+(2 \lambda ^n \sigma ^n \tilde{\sigma }^{3 m} \beta _1 b_2  \bar{x}
+\lambda ^n \sigma ^n \tilde{\sigma }^{2 m}  \tilde{\zeta }^m  \gamma _1 b_4 \bar{x}
+2 \sigma ^n  \tilde{\sigma }^{2 m} \beta _2 b_2  \bar{y}
) \tilde{H}_2(\tilde{\mathbf{x}}_{n})&\\
&
+(2 \lambda ^n \sigma ^n  \tilde{\sigma }^m \tilde{\zeta }^{2 m}\gamma _1 b_3  \bar{x}
+\lambda ^n \sigma ^n \tilde{\sigma }^{2 m} \tilde{\zeta }^m \beta _1 b_4  \bar{x}
+\sigma ^n   \tilde{\sigma }^m  \tilde{\zeta }^m \beta _2 b_4 \bar{y})\tilde{H}_3(\tilde{\mathbf{x}}_{n})&\\
&
+\sigma ^{2 n} \tilde{\sigma }^{4 m} b_2 \tilde{H}_2(\tilde{\mathbf{x}}_{n})^2
+\sigma ^{2 n}  \tilde{\sigma }^{3 m} \tilde{\zeta }^m b_4 \tilde{H}_2(\tilde{\mathbf{x}}_{n}) \tilde{H}_3(\tilde{\mathbf{x}}_{n})&\\
&
+\sigma ^{2 n}  \tilde{\sigma }^{2 m} \tilde{\zeta }^{2 m} b_3 \tilde{H}_3(\tilde{\mathbf{x}}_{n})^2
+\sigma ^{2 n} \tilde{\sigma }^{2 m} H_2(\mathbf{x}_{m}),&
\end{aligned}
\end{equation}
\begin{equation}\label{mn-dependent return map3}
\begin{aligned}
\bar{z}_{1}
&
=
(\lambda ^n  \tilde{\lambda }^m c_1 \alpha _1
+\lambda ^n  \tilde{\sigma }^m c_2 \beta _1)\bar{x}
+(\tilde{\lambda }^m \tilde{\sigma }^{-m} c_1 \alpha _2
+c_2 \beta _2) \bar{y}
+\tilde{\lambda }^m \zeta ^n  c_1 \alpha _3 \bar{z} &\\
&
+  \tilde{\lambda }^m \sigma ^n  \tilde{\sigma }^m c_1 \tilde{H}_1(\tilde{\mathbf{x}}_{n})
+\sigma ^n \tilde{\sigma }^{2 m} c_2 \tilde{H}_2(\tilde{\mathbf{x}}_{n})
+\sigma ^n \tilde{\sigma }^m H_3(\mathbf{x}_{m}).&
\end{aligned}
\end{equation}

Now, we check the convergence of
 (\ref{mn-dependent return map1})--(\ref{mn-dependent return map3})
as  $m, n$  tend to infinity.
First, we estimate  the higher order terms containing $\tilde{H}_i$ and $H_i$. 

\smallskip

$\bullet$ \emph{Higher order terms containing $\tilde{H}_1$, $\tilde{H}_2$, $\tilde{H}_3$}.
The coordinate $\bar{x}_{1}$ in  (\ref{mn-dependent return map1}) has three higher order terms containing $\tilde{H}_1, \tilde{H}_2$ and $\tilde{H}_3$ whose $(m,n)$-dependent 
 coefficients are  respectively
$$\tilde{\lambda }^m \sigma ^n \tilde{\sigma }^m,\
\sigma ^n \tilde{\sigma }^{2 m},\
 \sigma ^n  \tilde{\sigma }^m \tilde{\zeta }^m.$$
By condition (\ref{eigenvalue condition 1}), we have 
\begin{equation}\label{hot1}
0<\tilde{\lambda }^m \sigma ^n \tilde{\sigma }^m, \sigma ^n \tilde{\sigma }^{2 m}<
\sigma ^n  \tilde{\sigma }^m \tilde{\zeta }^m.
\end{equation}
As noted above, 
every $\tilde{H}_i$ has order 
 at least two. Thus,
 it is enough to check  the convergence to $0$ 
of   $\sigma ^n  \tilde{\sigma }^m \tilde{\zeta }^m \tilde{H}_i(\tilde{\mathbf{x}}_{n})$.
From (\ref{Landau-notation2}),  
\begin{equation}\label{hot2}
  \sigma ^n  \tilde{\sigma }^m \tilde{\zeta }^m \tilde{H}_i(\tilde{\mathbf{x}}_{n}) 
=O(\sigma^n  \tilde{\sigma}^m \tilde{\zeta }^m \zeta^{2n})
+
O(\sigma^{-n}  \tilde{\sigma}^{-3m} \tilde{\zeta }^m).
\end{equation} 
Since $ \vert m-nk -  \tilde{k} \vert<1$ by 
Lemma \ref{lem.Neutrality},  we have 
$$
\sigma^n  \tilde{\sigma}^m \tilde{\zeta}^m \zeta^{2n}<
 C (\sigma  \tilde{\sigma}^k \tilde{\zeta}^k \zeta^{2})^{n},\ 
 \sigma^{-n}  \tilde{\sigma}^{-3m} \tilde{\zeta}^m <
C^{\prime} (\sigma^{-1}  \tilde{\sigma}^{-3k} \tilde{\zeta}^{k})^{n} 
$$
where $C, C^{\prime}>0$ are constants independent of $m$ and $n$. 
Conditions (\ref{eigenvalue condition 2}) implies that 
the expression 
in (\ref{hot2}) converges to $0$ 
as $m, n\to +\infty$.

Moreover, 
 using  (\ref{eigenvalue condition 1})  and  (\ref{Landau-notation2}),
one can easily check that 
the $r$th order derivatives of 
$\sigma ^n  \tilde{\sigma }^m \tilde{\zeta }^m \tilde{H}_i(\tilde{\mathbf{x}}_{n})$ also 
converge to $0$ as $m, n\to +\infty$.
Thus, these higher order terms also converge to $0$ 
as $m, n\to +\infty$ in the $C^{r}$ topology 
on compact domains. 
\smallskip

We now evaluate  the expression of $\bar{y}_{1}$ in  (\ref{mn-dependent return map2}), which has
higher order terms containing $\tilde{H}_1, \tilde{H}_2, \tilde{H}_3, \tilde{H}_2^{2}, \tilde{H}_2\tilde{H}_3, \tilde{H}_3^{2}$.
Thus, 
let us see the convergences of terms containing $\tilde{H}_1$, $\tilde{H}_2$ and $\tilde{H}_3$.
 It is enough to study 
the following $(m,n)$-dependent terms
$$\tilde{\lambda }^m \sigma ^{2 n}  \tilde{\sigma }^{2 m} \tilde{H}_i, \
\lambda^n \sigma ^n \tilde{\sigma }^{3 m}  \tilde{H}_i, \
 \sigma^n  \tilde{\sigma }^{2 m}  \tilde{H}_i, \
\lambda^n \sigma^n  \tilde{\sigma}^m \tilde{\zeta}^{2 m}  \tilde{H}_i.$$
Note that the estimate  of 
$\sigma^n  \tilde{\sigma }^{2 m}\tilde{H}_i$ was  already done, see  (\ref{hot1}). 
Moreover, since  
$\lambda ^n \sigma ^n  \tilde{\sigma }^m \tilde{\zeta }^{2 m}
=(\lambda ^n\tilde{\zeta }^{m})\sigma ^n  \tilde{\sigma }^m \tilde{\zeta }^{m}$ 
and 
 $\lambda ^n\tilde{\zeta }^{m}$ converges to a constant by Lemma \ref{lem.Neutrality}, 
the convergence of
$\lambda^n \sigma^n  \tilde{\sigma}^m \tilde{\zeta}^{2 m} \tilde{H}_i$ follows from the discussion  above, 
see
(\ref{hot2}). 
So we have only to check the 
convergence of
$\tilde{\lambda }^m \sigma ^{2 n}  \tilde{\sigma }^{2 m} \tilde{H}_1(\tilde{\mathbf{x}}_{n})$
and
$\lambda ^n \sigma ^n \tilde{\sigma }^{3 m} \tilde{H}_2(\tilde{\mathbf{x}}_{n})$.
By (\ref{Landau-notation2}), \begin{align*}
& 
\tilde{\lambda }^m \sigma ^{2 n}  \tilde{\sigma }^{2 m} \tilde{H}_1(\tilde{\mathbf{x}}_{n}) 
=O(\tilde{\lambda }^m\tilde{\sigma}^{-2m})
+O( \tilde{\lambda }^m \sigma ^{2 n}  \tilde{\sigma }^{2 m}  \zeta^{2n}),\\
& 
\lambda ^n \sigma ^n \tilde{\sigma }^{3 m} \tilde{H}_2(\tilde{\mathbf{x}}_{n}) 
=O(\lambda ^n \sigma ^{-n} \tilde{\sigma }^{-m})
+O(\lambda ^n \sigma ^n \tilde{\sigma }^{3 m} \zeta^{2n}).
\end{align*}
By a similar discussion as above, using
$ \vert m-nk -  \tilde{k} \vert<1$ of Lemma \ref{lem.Neutrality}
and (\ref{eigenvalue condition 2}), 
these terms also converge to $0$ as $m, n\to +\infty$. 

Moreover, same as in the case of $\bar{x}_{1}$, one can check 
the convergence of $r$th order derivatives of $\tilde{\lambda }^m \sigma ^{2 n}  \tilde{\sigma }^{2 m} \tilde{H}_1(\tilde{\mathbf{x}}_{n})$
and
$\lambda ^n \sigma ^n \tilde{\sigma }^{3 m} \tilde{H}_2(\tilde{\mathbf{x}}_{n})$.
Hence, these higher order terms converge to $0$ 
as $m, n\to +\infty$ in the $C^{r}$ topology.

\smallskip

It remains to check the 
convergence of $(m,n)$-dependent terms 
involving 
 $\tilde{H}_1$ and $\tilde{H}_2$ in the
expression of $\bar{z}_{1}$.
These estimations follow from  (\ref{hot1}). Thus, 
we finish checking the convergence for higher order terms containing 
$\tilde{H}_1$, $\tilde{H}_2$ and $\tilde{H}_3$. 
\smallskip

$\bullet$ \emph{Higher order terms containing $H_1$, $H_2$, $H_3$}.
We start by evaluating the expressions of $\bar x_1$ and $\bar z_1$.
From  (\ref{eq.transition1}), $H_1$ and $H_3$ start from degree two
terms. 
Hence, the orders of $H_1$ and $H_3$ are dominated by 
those of $x_{m}^{2}, y_{m}^{2}, z_{m}^{2}, x_{m}y_{m}, y_{m}z_{m}, z_{m}x_{m}$, 
which can be evaluated directly from (\ref{Landau-notation3}). 
 In this way,  one can check 
the terms containing $H_1$ and $H_3$ in the expressions of $\bar{x}_{1}$ 
and  $\bar{z}_{1}$ 
in  (\ref{mn-dependent return map1}) and (\ref{mn-dependent return map3}) 
as follows:  for $i=1,3$, 
$$
\sigma^{n}\tilde{\sigma}^{m}H_i(\mathbf{x}_{m})
=
O(
\tilde{\lambda}^{2m}\sigma^{n}\tilde{\sigma}^{m}
) +
O(\sigma^{-3n}\tilde{\sigma}^{-5m})+
O(
\sigma^{n}\tilde{\sigma}^{3m}\zeta^{4n}
)+ O(\sigma^{-n}\tilde \sigma^{-m}).
$$
This converges to $0$ under the conditions  (\ref{eigenvalue condition 1})-(\ref{eigenvalue condition 2}) 
as $m, n\to \infty$. Moreover,  this and   (\ref{Landau-notation3}) imply that 
the $r$th order derivatives of $\sigma^{n}\tilde{\sigma}^{m}H_i(\mathbf{x}_{m})$ also 
converge to $0$ as $m, n\to +\infty$.

Next, we evaluate the
expressions of $\bar{y}_{1}$ which only contains
 terms  $H_2$. 
From (\ref{eq.transition1}),
$H_2(\mathbf{x}_{m})$ 
starts from the terms of
$x_{m}^{2}, x_{m}y_{m}$ and $z_{m}x_{m}$. 
From (\ref{Landau-notation3}) with (\ref{eigenvalue condition 2}), 
we only  have to check the  order of  terms for $x_{m}y_{m}$ and $z_{m}x_{m}$.
Note that from (\ref{Landau-notation3}),
$$
\sigma ^{2n} \tilde{\sigma }^{2m}  x_{m}y_{m}=
O(\tilde{\lambda}^{m}\tilde{\sigma}^{-m})+O(\tilde{\lambda}^{m}\sigma^{2n}\tilde{\sigma}^{3m}\zeta^{2n}),\ 
\sigma^{2n} \tilde{\sigma}^{2m}  z_{m}x_{m}=O(\tilde{\lambda}^{m}\sigma^{n}\tilde{\sigma}^{m}),
$$
both of which  converge to $0$ under  (\ref{eigenvalue condition 1})--(\ref{eigenvalue condition 2}) 
as $m, n\to \infty$.
This implies that $\sigma ^{2n} \tilde{\sigma}^{2m}  H_{2}$ 
and its  $r$th order derivatives
tend to $0$ as $m, n\to +\infty$.

\smallskip
$\bullet$ \emph{Conclusion.}
Recall that  
$\lambda^{n} \tilde{\zeta}^{m}$ converges to  $(\gamma_{1}  a_{3})^{-1}\xi$.
Hence,  due to the above evaluations, 
(\ref{mn-dependent return map1})--(\ref{mn-dependent return map3})  converge to
the following
\begin{equation}\label{intact-limit return map}
(\bar{x}, \bar{y},  \bar{z})\longmapsto
(
\xi \bar{x}+a_{2}\beta_{2}\bar{y},\
\bar{\mu}+b_{3}(\xi a_{3}^{-1})^{2}\bar{x}^{2}+ \beta_{2}^2  b_{2}\bar{y}^2
+\xi a_{3}^{-1}  \beta_{2} b_{4} \bar{x}\bar{y},\
c_{2}\beta_{2}\bar{y}
)
\end{equation}
in the $C^{r}$ topology on compact domains.
\smallskip

Finally, to improve the appearance of (\ref{intact-limit return map}),
 consider the next coordinate change 
\begin{equation}\label{coordinate change1}
\Theta: (\bar{x}, \bar{y}, \bar{z}, \bar{\mu})\longmapsto
 (\beta _2 a_2^{-1} b_2 \bar{x},\
 \beta _2^2 b_2 \bar{y},\
  \beta _2c_2^{-1} b_2 \bar{z},\
  \beta _2^2 b_2\bar{\mu})=:(\tilde{x}, \tilde{y}, \tilde{z}, \tilde{\mu}).
\end{equation}
By $\Theta$, (\ref{intact-limit return map})   is conjugate to
\begin{equation}\label{alt-limit return map}
(\tilde{x}, \tilde{y},  \tilde{z})\longmapsto
(
\xi  \tilde{x} +\tilde{y},\
\tilde{\mu} + \tilde{y}^{2}
+
\kappa_{1} \tilde{x}^{2}
+
\kappa_{2} \tilde{x} \tilde{y},\
\tilde{y}
),
\end{equation}
where (remembering the condition  (\ref{non-zero}))
\begin{equation}\label{kappa-conditions}
\kappa_{1}=(\xi a_2 a_3^{-1})^{2} b_2^{-1} b_3,\quad
\kappa_{2}=\xi a_2 (a_3 b_2)^{-1} b_4.
\end{equation}
This ends  the proof of Theorem \ref{thm.main1}-(3).
\hfill{$\square$}

\section{Robust connections between saddles of different indices}  \label{s.proof-of-last-thm}
In this section, we prove Theorem~\ref{thm.main3}, that is, 
we give the perturbation to obtain the connection between the blender and the saddle $Q$.

\subsection{Further perturbation}
Let  $\{f_{\mu,\nu}\} = \{f_{\mu_k(\bar{\mu}), \nu_k}  \}$ be the family in Theorem \ref{thm.main1}.
We consider another additional $n$-dependent perturbation $\{g_n\}=\{g_{n_{k}}\}$ of $\{ f_{\mu,\nu} \}$ 
in a small neighborhood of $X=(0,1,0)$ in $U_{Q}$ to show Theorem \ref{thm.main3}. 
In the following,  to simplify the notations,
we again drop the subscript $k$ of $m_k$ and $n_k$.

Let us consider a smooth bump function $b$ satisfying the following:
\begin{itemize}
\item $b(t) = 0$ for $|t| > 1/2$.
\item $b(t) = 1$ for $|t| < 1/3$.
\item $0 \leq  b(t) \leq 1$.
\end{itemize}
Then, given $n > 0$, consider the functions 
$B, B_n : \mathbb{R}^3 \to \mathbb{R}$ defined as follows:
\begin{itemize}
\item $B(x, y, z) := b(x) b(y) b(z)$,
\item $B_n(x, y, z) := \lambda^{n} B(x/\zeta^{n}, y/\zeta^{n}, z/\zeta^{n})$,
\end{itemize}
and an $n$-dependent neighborhood of $X=(0,1,0)$ defined as follows:
$$U_{X, n} =U_{X} =\left\{
(x,1+y,z)\ \vert \
|x|, |y|, |z|<\zeta^{n}/2.
\right\}.
$$
Note that $U_X$ converges to $\{ X\}$ as $n \to +\infty$.
We define 
 $\theta_{n} : M \to M$ as follows:
\begin{itemize}
\item if $(x,1+y,z)\in U_{X}$,
$$
\theta_{n}(x,1+y,z):=(x,1+y,z)+(B_{n}(x,y,z), 0, 0);
$$
\item otherwise, $\theta_{n}$ is the identity map.
\end{itemize}
Finally, we define
\begin{equation}\label{perturbation}
g_n:= \theta_{n}\circ  f_{\mu_k(\bar{\mu}), \nu_k}.
\end{equation}

For this perturbation, we prove the following:
\begin{lemma}
Suppose $\lambda/ \zeta^{1+\alpha} < 1$. 
Then the $C^{1+\alpha}$ distance 
between $g_{n_k}$ and 
$f_{\mu, \nu} = f_{\mu_k(\bar{\mu}), \nu_k}$ goes to zero as $k$ (hence $n_k$)
tends to $+\infty$.
\end{lemma}

Note that the condition $\lambda/ \zeta^{1+\alpha} < 1$ is 
equivalent to
$\alpha < \frac{\log\lambda}{\log \zeta} -1$.

\begin{proof}
The convergence in the  $C^0$ topology is easy to see.
To see the $C^{1+\alpha}$  convergence, 
we only need to check the 
$C^{1+\alpha}$ smallness of each partial derivative of $B_n$ for large $n$.
By the symmetry, we only confirm for $(B_n)_x$ 
(partial derivative of $B_n$ with respect to $x$)
and omit the check for $(B_n)_y$ and $(B_n)_z$.

The partial derivative $(B_n)_x$ is given as follows:
\[
(B_n)_x =
(B_n)_x (x,y,z):=
\left( \frac{\lambda}{\zeta} \right)^n 
b'\! \left(\frac{x}{\zeta^{n}}\right) 
b\! \left(\frac{y}{\zeta^{n}}\right)
b\! \left(\frac{z}{\zeta^{n}}\right).
\]

Then, since 
$\lambda/\zeta  < \lambda/\zeta^{1+\alpha} < 1$, 
we can see the the $C^1$ smallness of this function for large $n$.

Let us confirm the smallness of $\alpha$-H\"{o}lder constant for large $n$.
Given $(x_0, y_0, z_0)$, $(x, y, z ) \in \mathbb{R}^3$,  we have 
\begin{align*}
&\frac{(B_n)_x(x_0+x, y_0 +y, z_0+z) - (B_n)_x(x_0, y_0, z_0) }{|(x, y, z)|^{\alpha}} \\
=&\left( \frac{\lambda}{\zeta^{1+\alpha}} \right)^n 
\frac{b' \!  \left( \frac{x_0+x}{\zeta^{n}}\right) 
b\! \left( \frac{y_0+y}{\zeta^{n}}\right)
b\!  \left( \frac{z_0+z}{\zeta^{n}}\right)
-b' \! \left( \frac{x_0}{\zeta^{n}}\right) 
b\!  \left( \frac{y_0}{\zeta^{n}}\right)
b\!  \left( \frac{z_0}{\zeta^{n}}\right)
}{|(x/\zeta, y/\zeta, z/\zeta)|^{\alpha}}.
\end{align*}
In the last formula,
$(\lambda/\zeta^{1+\alpha})^n$ 
converges to zero as $n \to +\infty$ if $(\lambda/\zeta^{1+\alpha}) <1$,
and the absolute value of the rest of the formula is bounded by 
the H\"{o}lder constant of $b'(x)b(y)b(z)$,
which does not depend on $n$ and $(x_0, y_0, z_0)$.
This shows that the $\alpha$-H\"{o}lder constant of 
$|(B_n)_x|$ converges to $0$ as $n$ tends to $+\infty$.
\end{proof}

This perturbation may give some effect on points in the  renormalizations. 
However, we can prove the following:
\begin{lemma}\label{l.boxpert}
The perturbation $\theta_n$ does not give any effect on the 
first return map of the renormalization. More precisely, for every $X \in \Phi_{m, n}(\Delta)$, we have 
$(g_n)^{i}(X) = (f_{\mu, \nu})^i(X)$ for $i \in [0, N_2+m+N_1+n]$.
\end{lemma}
In particular, the blender of $f_{\mu, \nu}$ is not affected  by the perturbation. 
Thus $g_n$ also has the same blender. 
\begin{proof}
Take $X \in \Phi_{m, n}(\Delta)$. The only possibility where $(g_n)^i(X)$ may 
get some effect from the perturbation is that $(g_n)^{n}(X)$ is contained in the 
support of the perturbation $U_{X}$. Thus, let us calculate the position of $(g_n)^{n}(X)$.

A point $X \in \Phi_{m, n}(\Delta)$ has the following form:
\[
(\sigma^{-n}\tilde{\sigma}^{-m}x +1, \sigma^{-2n}\tilde{\sigma}^{-2m}y + \sigma^{-n}, 
\sigma^{-n}\tilde{\sigma}^{-m}z +1),
\]
where the point $(x, y, z)$ is chosen from some compact domain.
Hence, the coordinate of $(f_{\mu, \nu})^n(X)$ is given as 
$$( \lambda^{n} \sigma^{-n}  \tilde{\sigma}^{-m} x  + \lambda^{n},\
\sigma^{-n} \tilde{\sigma}^{-2m} y +1,\
\zeta^{n}\sigma^{-n}\tilde{\sigma}^{-m} z +\zeta^{n}).$$

We show that, if $n$ is sufficiently large, then
the $z$-coordinate of this point is greater than $\zeta^n/2$, 
which is the $z$-coordinate of the boundary of $U_X$. This implies that 
the point $(g_n)^{n}(X)$ is outside the support of the perturbation.

To see this, we take the quotient of these two quantities:
$$
\frac{\zeta^n /2}{\zeta^{n}\sigma^{-n}\tilde{\sigma}^{-m} z +\zeta^{n}}
 = \frac{1}{2(\sigma^{-n}\tilde{\sigma}^{-m} z +1)}.
$$
If $z$ is bounded, then by taking sufficiently large $m$ and $n$, 
one can check that the last term lies in the interval $(0, 1)$, which 
completes the proof.
\end{proof}

\subsection{Proof of Theorem \ref{thm.main3}}

The perturbation above keeps 
$Q_{g}=Q$ intact.
Meanwhile, 
the local unstable manifold  of $Q$ is no longer a straight line in $U_{Q}$.
Indeed, $W_{\loc}^{u}(Q)$ has a bumped sub-arc near $X$ which contains
the vertical part given as
$$\ell=\ell_{m,n}:=\{(\lambda^{n}, 1+\sigma^{-n}\tilde{\sigma}^{-2m} y, 0) \ \vert\ 
|y|<4\beta_2^2b_2 \},$$
see Figure \ref{fig7}.
Let us consider another segment
$$
\hat{\ell}=\hat{\ell}_{m,n}:=\left\{
\Phi_{m,n}(0,y,0) \ \vert\  |y|<4
\right\},
$$
which is contained in the boundary of the box
\begin{equation}\label{Delta-m-n}
\Delta_{m,n}:=\Phi_{m,n}(\Delta), 
\end{equation}
where $\Phi_{m,n}:=\Psi_{m,n}\circ \Theta\circ \tilde{\Theta}$, see 
 (\ref{coordinate change}),   (\ref{coordinate change1}) and (\ref{coordinate change2}) for definitions
and
$\Delta$ is the box given in (\ref{parameters/box for blenders}).
In the next proposition,
we compare the $g^{N_2 + m + N_1}$-images of $\ell$ and $\hat{\ell}$.
\begin{figure}[hbt]
\centering
\scalebox{0.82}{\includegraphics[clip]{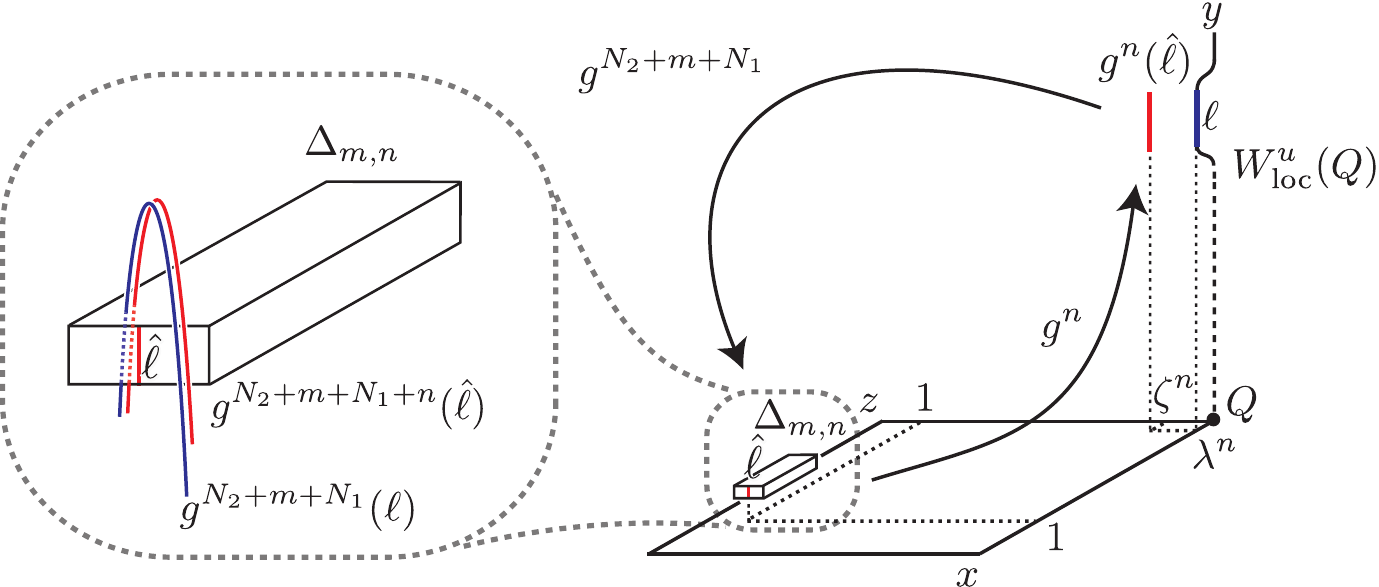}}
\caption{}
\label{fig7}
\end{figure}

\begin{proposition}\label{prop.sec5-prop1}
The $C^1$ distance between the
$(g^{N_{2}+m+N_{1}})$-images of $\ell$ and that of $g^{n}(\hat{\ell})$,
measured in the $\Phi_{m, n}$-coordinate, 
becomes arbitrarily small by letting $k$ (hence $(m, n)$) large (remember  that $k$ is the subscript for  renormalizations in Theorem \ref{thm.main1}).
\end{proposition}

By Remark~\ref{rem.bleint}, 
$g_{N_2+m+N_1+n}(\hat{\ell})$ has non-empty intersection 
with the stable manifold of the blender $\Lambda_g$. Hence, 
Proposition~\ref{prop.sec5-prop1} implies that in the perturbed system
the unstable manifold $W^u(Q)$ has non-empty intersection with the
superposition region of the blender $\Lambda_g$ as a vertical segment
for sufficiently large $m, n$. This concludes Theorem~\ref{thm.main3}.  

Thus, we only need to prove Proposition~\ref{prop.sec5-prop1}.
Let us start the proof.
\begin{proof}
The proof is obtained by similar calculations 
in Theorem \ref{thm.main1}-(2). 
In the following, we give explicit calculations of image of two segments 
with respect to $\Psi_{m, n}$-coordinate and prove the smallness 
of the $C^1$ distance of these two segments. Since the coordinate 
change $\Theta\circ\tilde{\Theta}$ is bounded (independent of $m$ and $n$),
it implies the $C^1$ smallness of the difference of two segments
in the $\Phi_{m, n}$-coordinate.
\smallskip

$\bullet$ \emph{Calculation of $\hat{\ell}$.}
First, for any $(0,t,0)\in \Phi^{-1}_{m,n}(\hat{\ell})$ where $|t|<4$,
we  put
$$(\hat{x}(t), \hat{y}(t), \hat{z}(t)):=
\Psi_{m,n}^{-1}
\circ g^{N_{2}+m+N_{1}+n}
\circ
\Phi_{m,n}(0,t,0).$$
Note that by Lemma~\ref{l.boxpert}, we know that $\hat{\ell}$ does not 
get any effect by the perturbation. Thus
by the same procedure as in  (\ref{mn-dependent return map1})-(\ref{mn-dependent return map3}),
each entry of $(\hat{x}(t), \hat{y}(t), \hat{z}(t))$ is given as follows:
\begin{align*}
& \hat{x}(t) =
 (\tilde{\lambda }^m \tilde{\sigma }^{-m} a_1 \alpha _2+ a_2 \beta _2)t+\sigma ^n \tilde{\lambda }^m \tilde{\sigma }^m a_1 \tilde{H}_1(\hat{\tilde{\mathbf{x}}}_{n}(t))+\sigma ^n \tilde{\sigma }^{2 m} a_2 \tilde{H}_2(\hat{\tilde{\mathbf{x}}}_{n}(t))\\
 &
 \hspace{20mm}
 +\sigma ^n \tilde{\zeta }^m \tilde{\sigma }^m a_3 \tilde{H}_3(\hat{\tilde{\mathbf{x}}}_{n}(t))+
\sigma ^n \tilde{\sigma }^m H_1(\hat{\mathbf{x}}_{m}(t));
\\
& \hat{y}(t)=
\bar{\mu}
+ \sigma ^n \tilde{\lambda }^m b_1 \alpha _2 t
+ \beta _2^2 b_2 t^2
+\sigma ^{2 n} \tilde{\lambda }^m \tilde{\sigma }^{2 m} b_1 \tilde{H}_1(\hat{\tilde{\mathbf{x}}}_{n}(t))
+2 t \sigma ^n \tilde{\sigma }^{2 m} \beta _2 b_2 \tilde{H}_2(\hat{\tilde{\mathbf{x}}}_{n} (t))\\
& \hspace{10mm}
+t \sigma ^n \tilde{\zeta }^m \tilde{\sigma }^m \beta _2 b_4 \tilde{H}_3(\hat{\tilde{\mathbf{x}}}_{n}(t))
+
\sigma ^{2 n} \tilde{\zeta }^m \tilde{\sigma }^{3 m} b_4 \tilde{H}_2(\hat{\tilde{\mathbf{x}}}_{n}(t)) \tilde{H}_3(\hat{\tilde{\mathbf{x}}}_{n}(t))
\\
&\hspace{15mm}
+\sigma ^{2 n} \tilde{\sigma }^{4 m} b_2 \tilde{H}_2(\hat{\tilde{\mathbf{x}}}_{n}(t))^2
+\sigma ^{2 n} \tilde{\zeta }^{2 m} \tilde{\sigma }^{2 m} b_3 \tilde{H}_3(\hat{\tilde{\mathbf{x}}}_{n}(t))^2
+\sigma ^{2 n} \tilde{\sigma }^{2 m} H_2(\hat{\mathbf{x}}_{m}(t));
\\
& \hat{z}(t) =
(\tilde{\lambda }^m \tilde{\sigma }^{-m} c_1 \alpha _2+ c_2 \beta _2)t
+\sigma ^n \tilde{\lambda }^m \tilde{\sigma }^m c_1 \tilde{H}_1(\hat{\tilde{\mathbf{x}}}_{n}(t))
\\
&\hspace{50mm}
+\sigma ^n \tilde{\sigma }^{2 m} c_2 \tilde{H}_2(\hat{\tilde{\mathbf{x}}}_{n}(t))+
\sigma ^n \tilde{\sigma }^m H_3(\hat{\mathbf{x}}_{m} (t)).
\end{align*}
Here,
$\hat{\tilde{\mathbf{x}}}_{n}(t)$ and $\hat{\mathbf{x}}_{m}(t)$ in the higher order terms are given as
$$
\hat{\tilde{\mathbf{x}}}_{n}(t)=g^{n}\circ\Phi_{m,n}(0,t,0)-(0,1,0),\
 \hat{\mathbf{x}}_{m}(t)=g^{m+N_{1}+n}\circ\Phi_{m,n}(0,t,0)-(0,1,1).
$$

$\bullet$ \emph{Calculation of $\ell$.}
Next, for any point
$(\lambda^{n}, 1+\sigma^{-n}\tilde{\sigma}^{-2m} t,0)\in \ell$ where
$|t|<4\beta_{2}^{2}b_{2}$,  let us calculate its image under $\Psi_{m,n}^{-1}
\circ g^{N_{2}+m+N_{1}}$.
First, if $m, n$ are sufficiently large, then $\ell$ is contained in the 
domain of the definition of the transition from $U_X$ to $U_Y$.
Thus $g^{N_1}(\ell) \subset U_P$.
Then, we calculate 
$g^{m+N_1}(\ell) := ({\tt x}(t), {\tt y}(t)+1, {\tt z}(t)+1)$. 
They are given as follows:  
\begin{align*}
& {\tt x}(t) =  \tilde{\lambda}^m 
+ t \sigma^{-n} \tilde{\sigma}^{-2m}\alpha_2 - \zeta^n\tilde{\lambda}^m\alpha_3
+ \tilde{\lambda}^m\tilde{H}_1(\tilde{\mathbf{x}}_{n}(t)),  \\
& {\tt y}(t) = t \sigma^{-n} \tilde{\sigma}^{-m}\beta_2 + 
\tilde{\sigma}^m\tilde{H}_2(\tilde{\mathbf{x}}_{n}(t)) ,\\
& {\tt z}(t) =  \tilde{\zeta}^m\tilde{H}_3(\tilde{\mathbf{x}}_{n}(t)).
\end{align*}
By a similar argument to obtain 
(\ref{Landau-notation3}), we can check that these three numbers converge to zero 
as $m, n$ tends to $+\infty$. Hence these points are in the domain 
of transition map from $U_P$ to $U_Q$, see also Remark~\ref{rem.finalcomment}
after the end of this proof. 

Finally, we calculate
$$
(x(t), y(t), z(t)):=
\Psi_{m,n}^{-1}
\circ g^{N_{2}+m+N_{1}} (\lambda^{n}, 1+\sigma^{-n}\tilde{\sigma}^{-2m} t,0).$$
By the same computation as the proof of Theorem~\ref{thm.main1}-(2), one has
\begin{align*}
& x(t) =
(\tilde{\lambda }^m \tilde{\sigma }^{-m} a_1 \alpha _2 +a_2 \beta _2) t
+
\sigma ^n \tilde{\lambda }^m \tilde{\sigma }^m a_1 \tilde{H}_1(\tilde{\mathbf{x}}_{n}(t))
+
\sigma ^n \tilde{\sigma }^{2 m} a_2 \tilde{H}_2(\tilde{\mathbf{x}}_{n}(t))
\\
&
\hspace{20mm}
+
\sigma ^n \tilde{\zeta }^m \tilde{\sigma }^m a_3 \tilde{H}_3(\tilde{\mathbf{x}}_{n}(t))
+
\sigma ^n \tilde{\sigma }^m H_1(\tilde{\mathbf{x}}_{m}(t))
-\tilde{\lambda }^m \zeta ^n   \tilde{\sigma }^m \sigma ^n  a_1 \alpha _3;
\\
& y(t) =
\bar{\mu}
+ \sigma ^n \tilde{\lambda }^m b_1 \alpha _2 t
+ \beta _2^2 b_2 t^2+\sigma ^{2 n} \tilde{\lambda }^m \tilde{\sigma }^{2 m} b_1 \tilde{H}_1(\tilde{\mathbf{x}}_{n}(t))
+2 t \sigma ^n \tilde{\sigma }^{2 m} \beta _2 b_2 \tilde{H}_2(\tilde{\mathbf{x}}_{n}(t))\\
& \hspace{20mm}
+t \sigma ^n \tilde{\zeta }^m \tilde{\sigma }^m \beta _2 b_4 \tilde{H}_3(\tilde{\mathbf{x}}_{n}(t))
+
\sigma ^{2 n} \tilde{\zeta }^m \tilde{\sigma }^{3 m} b_4 \tilde{H}_2(\tilde{\mathbf{x}}_{n}(t)) \tilde{H}_3(\tilde{\mathbf{x}}_{n}(t))
\\
&\hspace{25mm}
+\sigma ^{2 n} \tilde{\sigma }^{4 m} b_2 \tilde{H}_2(\tilde{\mathbf{x}}_{n}(t))^2
+
\sigma ^{2 n} \tilde{\zeta }^{2 m} \tilde{\sigma }^{2 m} b_3 \tilde{H}_3(\tilde{\mathbf{x}}_{n}(t))^2 \\
&\hspace{30mm}
+\sigma ^{2 n} \tilde{\sigma }^{2 m} H_2(\tilde{\mathbf{x}}_{m}(t))
-\tilde{\lambda }^m \zeta ^n \tilde{\sigma }^{2 m} \sigma ^{2 n}   b_1 \alpha _3;
\\
& z(t) =
(\tilde{\lambda }^m \tilde{\sigma }^{-m} c_1 \alpha _2+ c_2 \beta _2) t
+\sigma ^n \tilde{\lambda }^m \tilde{\sigma }^m c_1 \tilde{H}_1(\tilde{\mathbf{x}}_{n}(t))
\\
&\hspace{20mm}
+\sigma ^n \tilde{\sigma }^{2 m} c_2 \tilde{H}_2(\tilde{\mathbf{x}}_{n}(t))+
\sigma ^n \tilde{\sigma }^m H_3(\tilde{\mathbf{x}}_{m}(t))
-\tilde{\lambda }^m \zeta ^n \tilde{\sigma }^m  \sigma ^n   c_1 \alpha _3,
\end{align*}
where
$$
\tilde{\mathbf{x}}_{n}(t)=(\lambda^{n}, \sigma^{-n}\tilde{\sigma}^{-2m} t,0),\
\tilde{\mathbf{x}}_{m}(t)=g^{m+N_{1}}(\lambda^{n}, 1+\sigma^{-n}\tilde{\sigma}^{-2m} t,0)-(0,1,1).
$$

Note that
convergence of
the higher order terms  
in the above $(\tilde{x}(t),\tilde{y}(t),\tilde{z}(t))$ 
and $(x(t),y(t),z(t))$
are already contained in the proof of Theorem \ref{thm.main1}-(2).
Then we have
$$
\left\|
(x(t),y(t),z(t))-(\tilde{x}(t),\tilde{y}(t),\tilde{z}(t))
\right\|_{C^{1}}\leqslant
O (\tilde{\lambda }^m \zeta ^n   \tilde{\sigma }^{2 m} \sigma ^{2 n}).
$$
The conditions  (\ref{eigenvalue condition 1})-(\ref{eigenvalue condition 2})
imply that
$\tilde{\lambda }^m \zeta ^n   \tilde{\sigma }^{2 m} \sigma ^{2 n}$ converges to $0$ as
$m,n\to \infty$.
\end{proof}

\begin{remark}\label{rem.finalcomment}
Let us discuss the importance of the perturbation $\theta_n$.
If we do not give the perturbation, then 
$g^{m+N_1}(\ell)$ is in general not contained 
in the domain of the transition $f^{N_2}$.
More precisely, without $\theta_n$, 
in the $y$-coordinate of $g^{m+N_1}(\ell)$
there remains  a term of the form $\lambda^n\tilde{\zeta}^m\beta_2$ 
which converges to some non-zero constant. 
The perturbation $\theta_n$ is performed so as to annihilate this term. 
\end{remark}

\noindent{{\bf{Acknowledgements.}}}
The authors thank Hiroshi Kokubu and Ming-Chia Li
for financial support which brought us together. 
We also thank
the warm hospitality of Kyoto University, NCTU, and PUC-Rio.
This paper was partially supported by GCOE (Kyoto University), 
JSPS KAKENHI Grant Number 22540226,
CNPq (Research and PDJ grants), Faperj (CNE),  Palis-Balzan project
and the Aihara Project, the FIRST program
from JSPS, initiated by CSTP.


\end{document}